\documentclass[11pt,a4paper]{article}
\usepackage[utf8]{inputenc}
\usepackage[english]{babel}
\usepackage{amsmath,amsthm,amsfonts,amssymb}
\usepackage{array,graphicx,xcolor,mathrsfs,dsfont,
subcaption,multicol,lipsum,bm}
\usepackage{geometry}
\usepackage{hyperref}
\usepackage{enumitem}
\hypersetup{
    linktoc=page,linkcolor=red,citecolor=blue,filecolor=blue,
urlcolor=cyan,colorlinks=true}
\theoremstyle{plain}
\newtheorem{theorem}{Theorem}[section]
\newtheorem{proposition}[theorem]{Proposition}
\newtheorem{lemma}[theorem]{Lemma}
\newtheorem{corollary}[theorem]{Corollary}

\newtheorem{claim}[theorem]{Claim}
\theoremstyle{definition}
\newtheorem{definition}[theorem]{Definition}
\newtheorem{remark}[theorem]{Remark}

\newcommand{\Z}{\mathbb{Z}}

\newcommand{\R}{\mathbb{R}}

\newcommand{\E}{\mathbb{E}}
\newcommand{\Pro}{\mathbb{P}}

\def\<#1{\langle #1\rangle}
\newcommand{\T}{\mathbb{T}}

\newcommand{\cit}{c_{\mathrm{FE}}}

\usepackage[normalem]{ulem}

\usepackage{mathtools}
\mathtoolsset{showonlyrefs}

\author{Vincent Tassion\thanks{ETH Zürich, Rämistrasse 101, 8092 Zürich, Switzerland, vincent.tassion@math.ethz.ch} \quad Hugo Vanneuville\thanks{CNRS and Institut Fourier (Université Grenoble Alpes), 100 rue des mathématiques, 38402 Saint-Martin-d'Hères, France, hugo.vanneuville@univ-grenoble-alpes.fr}
}

\title{Noise sensitivity of crossings for high temperature Ising model}

\date{}

\begin{document}

\maketitle

\abstract{Consider the event that there is a $+$ crossing from left to right in a box for the Ising model on the triangular lattice. We show that this event is noise sensitive under Glauber dynamics $t \mapsto \sigma_t$ in the subcritical regime $\beta<\beta_c$. We rely on the non-spectral approach from our previous work \cite{TV23}. An important aspect in this more general setup is the study of the pair $(\sigma_0,\sigma_t)$ and in particular the establishment of properties such as finite-energy and spatial mixing.}

\tableofcontents

\section{Introduction}

\subsection{Noise sensitivity}\label{ssec_noise}

The concept of noise sensitivity was introduced by Benjamini, Kalai and Schramm in \cite{BKS99} for  critical planar Bernoulli  percolation. In this model,  each site of the triangular lattice is first colored  black or white independently with probability $1/2$. Then, the resulting configuration evolves in time by resampling each site at rate $1$.  
They considered the event that a large rectangle is crossed from left to right and proved that this event  is noise sensitive:  for every fixed $t>0$ (even very small), its occurrence at time $0$ and time $t$ are asymptotically independent. A refinement of  the result was later obtained by Schramm and Steif \cite{SS10} using randomized algorithms. Then, Garban, Pete and Schramm \cite{GPS10} obtained a precise description of the phenomenon by identifying the  amount of noise implying asymptotic independence.


\smallskip

Noise sensitivity has also been proven for other planar critical percolation models: Bernoulli percolation under exclusion dynamics \cite{BGS13,GV19a}, dynamical Bernoulli percolation in a Liouville environment \cite{GHSS19}, Poisson Boolean percolation \cite{ABGM14}, Voronoi percolation \cite{AGMT16,AB18,Van19} and percolation of nodal lines \cite{GV19b}. We refer to the book \cite{GS14} for more about noise sensitivity and its motivations.
All of these results rely on a study of the dynamics via Fourier methods. In  \cite{TV23}, we gave a new  proof of the sharp noise  sensitivity theorem of \cite{GPS10} by relying on differential inequalities, inspired by the work of Kesten \cite{Kes87} on near-critical percolation. One motivation was to provide a more robust approach than the Fourier method. In the present paper, we apply it to high temperature Ising model evolving under Glauber dynamics.

\subsection{Noise sensitivity for high temperature Ising model}\label{ssec:1.2}

Consider the regular planar triangular lattice and let $\T$ denote the set of sites of this lattice, that is, $\T = \{n+me^{i\pi/3} : n,m \in \Z\}$. If $x,y \in \T$, we write $x \sim y$ to say that $\{x,y\}$ is an edge of the triangular lattice, i.e.\ $|x-y|=1$, where $|\cdot|$ is the Euclidean norm.
Let $n\ge 1$ and $x\in \T$. We consider the rhombus
\[
\Lambda_n = \big\{ k+me^{i\pi/3}\in \T: k,m \in \{-n,\dots,n\}\big\}.
\]
If $x \in \T$, we let $\Lambda_n(x)=x+\Lambda_n$ be the rhombus centered at $x$. When we say rhombus in this paper, we always refer to a set of this type. 
 Given a rhombus $\Lambda$ and some $\delta>0$, we let $\Lambda^{(\delta)}$ denote the `thickened rhombus' whose center is the same as $\Lambda$ and whose sides have been scaled by factor $1+\delta$ (with the convention that here and in all the paper we consider the ceil of the new size of sides to work with integer-valued sides).
 
\medskip

Let
\[
\Omega_n=\{-1,+1\}^{\Lambda_n}
\]
denote the set of all spin configurations in $\Lambda_n$.

\begin{definition}[Ising measure with free boundary conditions]\label{defi:ising}
Let $n\ge 1$ and let $E(\Lambda_n)$ denote the set of edges whose endpoints are both in $\Lambda_n$. The energy of a spin configuration $\eta\in\Omega_n$ is
\[
H(\eta)=-\sum_{\{x,y\} \in E(\Lambda_n)} \eta(x)\eta(y).
\]
Given some $\beta\in(0,+\infty)$, the Ising measure at inverse temperature $\beta$ (with free boundary conditions and no external field) on $\Lambda_n$ is the probability measure $\mu_n=\mu_{n,\beta}$ on $\Omega_n$ defined by
\[
\forall \eta\in\Omega_n \quad \mu_n(\eta)=\frac{1}{Z_n}\exp\big(-\beta H(\eta)\big),
\]
where $Z_n=Z_{n,\beta}$ is the renormalizing constant $\sum_{\eta\in\Omega_n}\exp(-\beta H(\eta))$.
\end{definition}
In this paper, we prove noise sensitivity properties of crossing events, that we define now.
\begin{definition}
For every $n \ge 1$, we let $\mathsf{Cross}_n \subset \Omega_{2n}$ denote the event that there is a path included in $\Lambda_n$ from the left side of this rhombus to its right side, that is only made of $+1$ spins.  We study the crossing of $\Lambda_n$ for spin configurations in the larger rhombus $\Lambda_{2n}$ to avoid   boundary effects.
\end{definition}
Let $\beta\in(0,+\infty)$, $n\ge 1$ and let $\sigma\sim \mu_{2n}=\mu_{2n,\beta}$. By symmetry of the rhombus and of the measure $\mu_{2n}$ and by self-duality, we have
\[
\Pro\big(\sigma \in \mathsf{Cross}_n \big) = \mu_{2n}\big(\mathsf{Cross}_n\big)= \frac{1}{2}.
\]
We consider a dynamics $(\sigma_s)_{s \ge 0}$ defined by sampling $\sigma_0=\sigma\sim\mu_n$  and letting it evolve according to Glauber resampling dynamics: we associate Poisson clocks to every site and, every time a clock rings, the value of the corresponding spin is resampled  (under the law of this spin conditionally on the value of the neighbouring spins). This dynamics is defined more precisely in Section \ref{ssec:setup}. 
The main result of our paper is that the crossing events in the Ising model at inverse temperature $\beta<\beta_c$ are sensitive under Glauber dynamics (see the appendix for the definition of~$\beta_c$).  
\begin{theorem}[Noise sensitivity]\label{thm:noise}
Let $\beta < \beta_c$. For any $t>0$,
\[
\Pro \big(\sigma,\sigma_t \in \mathsf{Cross}_n\big) \underset{n\rightarrow+\infty}{\longrightarrow} 1/4,
\]
where $(\sigma_s)_{s \ge 0}$ is the Glauber dynamics in $\Lambda_{2n}$ (with free boundary conditions), at inverse temperature $\beta$ and starting from $\sigma\sim \mu_{2n}$ (as defined in Section \ref{ssec:setup}).
\end{theorem}
\begin{samepage}
  \noindent\textbf{Remarks:}
  \begin{itemize}
  \item We stress that even if the model is subcritical from the point of view of spin/spin correlations (i.e.\ $\beta<\beta_c$), it is critical from the point of view of percolation observables (for instance, crossing probabilities are bounded away from $0$ and $1$, see Section \ref{ssec:perco}).
  \item Readers are probably wondering about the cases $\beta = \beta_c$ and $\beta > \beta_c$. In both these cases, the crossing events are not noise sensitive anymore.\footnote{At least they are not noise sensitive at a fixed time $t>0$. The question of rather studying the sensitivity at a time that equals a constant times the relaxation time might be a more natural question.} The reason is that the exponent of the four-arm event (see Definition \ref{defi:4arm}) is larger than $2$ in these cases. More precisely, if we use the notations from Section \ref{ssec:sharp}, then $\alpha_n \le n^{-(2+c)}$ for some $c>0$ if $\beta=\beta_c$ (see the work of Wu \cite{Wu18}) and $\alpha_n$ even decays exponentially fast if $\beta > \beta_c$. The general differential formulas (see e.g.\ Section~\ref{ssec:diff}) that we use in this paper then imply that, for every $t>0$, $\Pro (\sigma,\sigma_t \in \mathsf{Cross}_n)\rightarrow 1/2$.
  \item We stated the result for the crossing event of a rhombus, but the proof adapts directly to more general shapes, for instance with $\Lambda_n$ replaced by a $\rho n \times n$ rectangle for some fixed $\rho > 0$.
  \item We have chosen to work in finite volume in the whole paper. Another possibility could have been to sample $\sigma$ according to the Gibbs measure $\mu$ at inverse temperature $\beta$ in infinite volume (which is unique when $\beta<\beta_c$)\footnote{See e.g.\ \cite{FV17} for the definition of this measure.} and let this configuration evolve according to the Glauber dynamics in infinite volume. We have preferred to work only in finite volume in order to be able to ignore several technical issues (in particular in the derivation of differential formulas). The (essentially unique) drawback is that translation invariance does not hold but is replaced by a `quasi-invariance by translation', see Section~\ref{ssec:inv_dyn}.
  \end{itemize}
\end{samepage}
Concerning other questions related to noise sensitivity in the Ising model, see in particular the work of Galicza and Pete \cite{GP24} (who prove that no sparse reconstruction is possible for transitive functions in the regime $\beta<\beta_c$) and the work of Broman and Steif \cite{BS05} (about stability properties in Ising models that are non-critical from the point of view of percolation observables).

\subsection{Sharp noise sensitivity}\label{ssec:sharp}

Our analysis is actually quantitative and we obtain the sharp noise sensitivity theorem, i.e.\ the analogue of the theorem obtained by Garban, Pete and Schramm \cite{GPS10} for Bernoulli percolation evolving under i.i.d.\ dynamics. To state this theorem, we need to define the $4$-arm event.
\begin{definition}\label{defi:4arm}
Let $\beta\in(0,+\infty)$ and $n\ge 1$. We let $A_4(n) \subset \Omega_{2n}$ denote the event that there are four paths from a neighbour of the origin $o:=(0,0)$ to $\partial\Lambda_n:=\{x\in \Lambda_n : \exists y \in \T\setminus\Lambda_n : y \sim x\}$, such that each path is of constant sign and the signs of the four paths alternate around~$o$. $A_4(n)$ is called the `$4$-arm event'. We let
\[
\alpha_n=\mu_{2n}(A_4(n)) \quad \text{and} \quad \varepsilon_n=1/(n^2\alpha_n).
\]
\end{definition}
\begin{remark}\label{rem:alpha_4}
It can help to keep in mind that, if $\beta<\beta_c$, then there exist $c,C>0$ such that $\varepsilon_n\le Cn^{-c}$ for every $n\ge 1$ (see Section~\ref{ssec:perco}).
\end{remark}
The noise sensitivity property Theorem \ref{thm:noise} is a consequence of the following theorem (and of Remark \ref{rem:alpha_4}).
\begin{theorem}[Sharp noise sensitivity]\label{thm:sharp}
Let $\beta<\beta_c$ and let $(t_n)_{n\ge 1}$ be a sequence of non-negative numbers.
\begin{itemize}
\item If $t_n / \varepsilon_n \rightarrow +\infty$, then $\Pro\big(\sigma,\sigma_{t_n} \in \mathsf{Cross}_n\big) \rightarrow 1/4$; 
\item If $t_n / \varepsilon_n \rightarrow 0$, then $\Pro\big(\sigma,\sigma_{t_n} \in \mathsf{Cross}_n\big) \rightarrow 1/2$;
\end{itemize}
where $(\sigma_s)_{s \ge 0}$ is the Glauber dynamics in $\Lambda_{2n}$ (with free boundary conditions), at inverse temperature $\beta$ and starting from $\sigma\sim \mu_{2n}$ (as defined in Section \ref{ssec:setup}).
\end{theorem}
The   main novelty in the theorem above is the first  item (noise region). Indeed, the second item (stability region) could be derived from the arguments in \cite[Section 8]{GPS18}. 

\subsection{Strategy of the proof}\label{ssec:high}

To prove the noise sensitivity theorems, we follow the approach  from our work \cite{TV23}, which in the present context consists in studying differential inequalities satisfied by $t \mapsto \Pro(\sigma,\sigma_t\in A)$ for percolation-type events $A$. We prove that, when $t$ is small enough, the pair $(\sigma,\sigma_t)$ essentially satisfies all the desired general properties to run 
 this proof. In particular:
\begin{itemize}
  \item (Finite-energy). Changing the spin value at times $0$ and $t$ at some fixed vertex has a cost bounded independently of the rest of the configuration.
  \item (Spatial mixing, when $\beta\le \beta_c$). The configuration in a box is quasi-independent of  the configuration at macroscopic distance from it. A typical application is about arm events: let $A_1(m,n)$ be the one-arm event, that is, the event that there is a path made of spins $+1$ from $\partial\Lambda_m$ to $\partial\Lambda_n$. If $n \ge 2m$ then
    \[
      \Pro(\sigma,\sigma_t \in A_1(1,n)) \le C\cdot \Pro(\sigma,\sigma_t \in A_1(1,m))\Pro(\sigma,\sigma_t \in A_1(2m,n)).
      \]      
\item (Russo-type differential formula). The derivative of $t \mapsto\Pro(\sigma,\sigma_t\in A)$ can be expressed in terms of pivotal events for the event $A$. This essentially follows from standard derivative formulas for Markovian dynamics, together with the above-mentioned finite-energy property.
  \end{itemize}
The first two properties are known for the static Ising model (i.e.\ for $\sigma$ instead of $(\sigma,\sigma_t)$) and in this case essentially follow from the spatial Markov property of this model. Such a direct approach seems delicate since the pair $(\sigma,\sigma_t)$ \textit{does not satisfy the  spatial Markov property}.\footnote{For more about this lack of spatial Markov property for the pair, see the sketch of proof from Section~\ref{sec:finite-energy-one}.} Instead, we decompose static and dynamical influences. More precisely the  spin at some site $x$ at time $0$ can  influence the  spin at some other site $y$ at time $t$  through two  mechanisms: Either the spin at $x$ has already influenced the spin at $y$ at time $0$, which can be estimated by using the properties of the static Ising model. Or the information of the value of the spin at $x$ has traveled through time,  which we estimate using  that the set of sites that have rung before time $t$ is a subcritical percolation (provided $t$ is small enough) and that the  dynamics is local (the probability that a spin changes when a clock rings only depends on the neighbours of the site).

\medskip

Finite-energy, spatial mixing and a Russo-type formula are proven in Sections \ref{sec:ins_tol}, \ref{sec:mix_gen} and \ref{sec:mix_sub}. Then, the percolation arguments inspired from \cite{TV23} are done in Sections \ref{sec:quasimulti} and \ref{sec:noise_sens}. Most of the work in these last sections consists of proving the so-called quasi-multiplicativity property for the quantities $\Pro[\text{a $4$-arm event holds at times $0$ and $t$}]$. The general strategy is standard, but new difficulties emerge due to possible interactions between $-1$ paths at time $0$ and $+1$ paths at time $t$. See Section \ref{sec:quasimulti} (and in particular the paragraph below Proposition \ref{prop:sep}).

\subsection{Setup}\label{ssec:setup}

In this section, we introduce the general setup we will be working with in the rest of the paper. More precisely, in Definition \ref{defi:glauber} below, we define a Glauber dynamics in $\Lambda_n$. Each time we refer to the `Glauber dynamics in $\Lambda_n$' in this paper, we refer to the dynamical process $(\sigma_s)_{s\ge 0}$ that is introduced in this definition. Let us fix some $\beta\in(0,+\infty)$ once and for all in this section and let us first introduce some notations:
\begin{itemize}[noitemsep]
\item If $n\ge 1$, $\eta\in \Omega_n$ and $x \in \Lambda_n$, we let $\eta^x \in \Omega_n$ be the spin configuration obtained from $\eta$ by just flipping the spin at $x$;
\item We define the `rates' $c_x : \Omega_n \rightarrow [0,1]$ by
\[
c_x(\eta)=\frac{\mu_n(\eta^x)}{\mu_n(\eta)+\mu_n(\eta^x)} = \frac{1}{1+\exp\big(2\beta\sum_{y \in \Lambda_n : y\sim x}\eta_x\eta_y \big)}.
\]
N.B: $c_x(\eta)$ only depends on the value of $\eta$ at $x$ and its neighbours and is equal to the probability that the spin at $x$ equals $-\eta(x)$ conditionally on the event that the spin configuration coincides with $\eta$ at the neighbours of $x$. 
\end{itemize}
 
\begin{remark}
The rates $c_x$ defined above correspond to the so-called Glauber heat-bath dynamics. All of our analysis would also work in the case of the so-called Glauber Metropolis dynamics, where $c_x(\eta)$ is replaced by $\min \{ 1, \mu_n(\eta^x)/\mu_n(\eta)\}$.
\end{remark}
\begin{definition}[Glauber dynamics]\label{defi:glauber}
Given some $n\ge 1$, we define a continuous time dynamical process $(\sigma_s)_{s \ge 0}$ on $\Omega_n$ called `Glauber dynamics in $\Lambda_n$' (with free boundary conditions and no external field). This process is defined on some probability space $(\Omega,\mathcal{F},\Pro)$ and is constructed as follows:
\begin{itemize}[noitemsep]
\item We sample $\sigma \sim \mu_n$ and we let $\sigma_0=\sigma$;
\item We associate a Poisson Point Process $Y_x$ of parameter $1$ on $\R_+$ (a `Poisson clock') to every $x \in \Lambda_n$, independently for every $x$ and independently of $\sigma$;
\item We let $(U_{x,i})_{x \in \Lambda_n,i\ge 1}$ be a family of independent variables that are uniform in $[0,1]$ and are independent of $\sigma$ and the Poisson clocks;
\item If the clock on a vertex $x$ has rung for the $i^{th}$ time at some time $t$, then we flip the corresponding spin if $U_{x,i}\ge 1- c_x(\sigma_{t-})$. 
\end{itemize}
We let $Y=(Y_x)_{x \in \Lambda_n}$ and $U=(U_{x,i})_{x \in \Lambda_n,i\ge 1}$.
\end{definition}
We have thus defined a continuous time càdlàg Markov process with values in $\Omega_n$. See for instance \cite[Chapter 4]{Lig85} more about this process. It is important to bear in mind that this process starts from its invariant measure $\mu_n$, and that this measure is reversible.

\paragraph{Some conventions about vocabulary and notations.}
In all the paper, given some $n\ge 1$, the sentence `$A$ is a static event' means that $A$ is a subset of $\Omega_n$ and, given some $V\subseteq\Lambda_n$, the sentence `$A$ is a static event that depends only on $V$' means that $A$ is a subset of $\Omega_n$ such that $1_A(\eta')=1_A(\eta)$ for every spin configurations $\eta,\eta'\in \Omega_n$ that coincide in $V$.
In order to make the difference between static and dynamical events, we use letters $A,B,\dots$ to denote static events, i.e.\ subsets of $\Omega_n$, and letters $\mathcal{A},\mathcal{B},\dots$ to denote dynamical events, i.e.\ elements of the $\sigma$-algebra $\mathcal{F}$ from Definition \ref{defi:glauber}.
We use the letter $\eta$ to denote a fixed spin configuration (i.e.\ an element of $\Omega_n$) and, as the reader can see in Definition \ref{defi:glauber}, $\sigma$ is a random spin configuration sampled according to $\mu_n$ and $(\sigma_s)_{s \ge 0}$ is the Ising Glauber dynamics in $\Lambda_n$ that starts from $\sigma$.

\paragraph{Boundaries.} If $V\subseteq\T$, we let $\partial V = \{x\in V : \exists y \in \T \setminus V, y \sim x \}$ and $\partial_\mathrm{ext}V=\{x\in \T\setminus V : \exists y \in V, y \sim x \}$.

\paragraph{Conventions about constants.}
In all the paper, the constants $c$, $C$, $c'$, $C'$, $c_0$, $C_0$ are positive finite constants that are independent of everything except the inverse temperature $\beta$. Sometimes, the constants will depend on a parameter $\delta$. In this case, we will use the notations $c_\delta,C_\delta$. 
In the proofs, we will often omit the constants and rather use the notations $\lesssim$ and $\asymp$: if $f,g$ are positive functions, we write $f \lesssim g$ if there exists $C>0$ such that $f \le Cg$ and we write $f \asymp g$ if $f \lesssim g$ and $g \lesssim f$.

\paragraph{Acknowledgments.} We thank Fabio Martinelli, Sébastien Ott and G\'{a}bor Pete for background on Glauber dynamics, Vincent Beffara, Loren Coquille and Corentin Faipeur for helpful discussions on the Ising model, and Malo Hillairet and Ritvik Radhakrishnan for useful comments on a preliminary version of this paper.

This project has received funding from the European Research Council (ERC) under the European Union’s Horizon 2020 research and innovation programme (grant agreement No 851565).

\section{Background}

\subsection{Background on static Ising model}\label{back_stat}

\subsubsection{Three properties}\label{ssec:bc}

Let us start this subsection with a preliminary remark: In the present section, we consider the Ising model on some general finite set $V\subset\T$ and with some general boundary conditions. We will consider the Ising model in this generality a few times but we will mostly work with the Ising measure $\mu_n$ with free boundary conditions on $\Lambda_n$, so the notations from this subsection will in fact only appear very occasionally.

Let $\beta\in(0,+\infty)$, let us also fix some finite set $V \subset \T$ and recall that $\partial_{\mathrm{ext}} V=\{x\in\T\setminus V : \exists y \in V, y\sim x\}$. Moreover, let $\xi$ be a boundary condition on $V$, that is, an element of $\{-1,0,+1\}^{\partial_{\mathrm{ext}} V}$. We let $\mu_V^\xi$ be the Ising probability measure on $\Omega_V:=\{-1,+1\}^V$ with boundary condition $\xi$ (see the appendix for the definition of this measure). We let $\langle\cdot\rangle_V^\xi$ denote the corresponding expectation. Moreover, if $\xi \equiv +1$, we let $\mu^\xi_V=\mu_V^+$, if $\xi\equiv -1$, we let $\mu^\xi_V=\mu_V^-$ and if $\xi \equiv 0$, we let $\mu_V^\xi=\mu_V^\emptyset$ (so $\mu_n=\mu_{\Lambda_n}^\emptyset$).

\medskip

We refer for instance to \cite{FV17} for a proof of the following three properties. Recall that we have fixed a finite subset $V\subset \T$. We say that a function $f : \Omega_V \rightarrow \R$ is increasing if $f(\eta')\ge f(\eta)$ for every $\eta,\eta'$ such that $\eta'(x)\ge \eta(x)$ for every $x\in V$. 

\paragraph{FKG inequality.} For any increasing functions $f,g : \Omega_V \rightarrow \R$ and any boundary condition~$\xi\in \{-1,0,+1\}^{\partial_{\mathrm{ext}} V}$,
\begin{equation}\label{eq:fkg}
\tag{FKG}
\< {fg}_V^\xi\ge \< f_V^\xi \< g_V^\xi.
\end{equation}

\paragraph{Comparison between boundary conditions.} Let $\xi,\psi\in\{-1,0,+1\}^{\partial_{\mathrm{ext}} V}$ be two boundary conditions on $V$ such that $\psi(x)\ge \xi(x)$ for every $x \in \partial_{\mathrm{ext}} V$. For any increasing function $f : \Omega_V \rightarrow \R$,
\begin{equation}\label{eq:mon}
\tag{MON}
\< f_V^{\psi}\ge \< f_V^\xi.
\end{equation}

\paragraph{Spatial Markov property.} Let $W \subseteq V$, let $\xi \in \{-1,0,+1\}^{\partial_\mathrm{ext}V}$ be a boundary condition on $V$ and let $\eta \in \{-1,+1\}^{V\setminus W}$. Moreover, let $\psi\in\{-1,0,+1\}^{\partial_\mathrm{ext}W}$ be the boundary condition on $W$ that coincides with $\xi$ in $ \partial_\mathrm{ext}W \cap \partial_\mathrm{ext}V$ and with $\eta$ in $(\partial_\mathrm{ext}W) \cap V$. We have
\begin{equation}\label{eq:smp}
\tag{SMP}
\mu_V^\xi( \text{spin config.\ in $W$ = }\cdot \mid \text{spin config.\ in $V\setminus W$ equals $\eta$} ) = \mu_{W}^{\psi}.
\end{equation}

\subsubsection{Percolation properties}\label{ssec:perco}

In this section, we state two key percolation properties for the (static) Ising model at inverse temperature $\beta\le \beta_c$ or $\beta<\beta_c$. These two properties are not new and are proven for completeness in the appendix. If $1 \le m \le n$, we consider the elongated rhombus $\Lambda_{m,n}=\{k+\ell e^{i\pi/3} : k \in \{-m,\dots,m\}, \ell \in \{-n,\dots,n\}\}$. Moreover, we let $\mathsf{Cross}_{m,n}$ denote the left-right crossing event of $\Lambda_{m,n}$ by $+1$ spins (so $\mathsf{Cross}_{n,n}=\mathsf{Cross}_n$).
\paragraph{Box crossing property.} Let $\beta\le \beta_c$. Then, for every $\rho>0$ and $\delta>0$, there exists $c_{\rho,\delta}>0$ such that, for every $n\ge (1/\rho) \vee 1$,
\begin{equation}\label{eq:bxp}
\tag{BXP}
\mu_{\Lambda^{(\delta)}_{\rho n,n}}^-(\mathsf{Cross}_{\rho n, n}) \ge c_{\rho,\delta},
\end{equation}
where $\Lambda^{(\delta)}_{\rho n,n}$ is the thickened elongated rhombus $\Lambda_{(1+\delta)\rho n,(1+\delta)n}$.
\paragraph{Upper bound on the `exponent' of the $4$-arm event.} Let $\beta < \beta_c$ and recall that $\alpha_n$ is the $\mu_{2n}$-probability of the $4$-arm event from the origin $o$ to $\partial \Lambda_n$. There exists $c>0$ such that, for every $n\ge m \ge 1$,
\begin{equation}\label{eq:4arm}
\frac{\alpha_n}{\alpha_m} \ge c\Big(\frac{m}{n}\Big)^{2-c}.
\end{equation}
\subsubsection{Spatial mixing property}\label{ssec:inv}
The first brick in the proof of the spatial mixing property for the pair $(\sigma,\sigma_t)$ at $\beta\le \beta_c$ is the analogous static property, that we state now. We do not write the proof of this property because it is exactly the same as the analogous result for FK percolation, see for instance \cite[Theorem~5]{DST17} and \cite[Corollary 2.10]{DM22} (Remark: the proof is based on \eqref{eq:bxp}, \eqref{eq:smp} and \eqref{eq:mon}).
\paragraph{Spatial mixing.} Let $\beta\le \beta_c$ and $\delta>0$. There exist $c_\delta,C_\delta>0$ such that the following holds for every rhombus $\Lambda$. Let $A \subseteq \{-1,+1\}^{\Lambda^{(\delta)}}$ that depends only on $\Lambda$ and let $\xi$ and $\psi$ be two boundary conditions on the thickened rhombus $\Lambda^{(\delta)}$, i.e.\ $\xi,\psi\in\{-1,0,+1\}^{\partial_\mathrm{ext}\Lambda^{(\delta)}}$. Then,
\begin{equation}\label{eq:stat_mix_initial}
c_\delta \cdot \mu_{\Lambda^{(\delta)}}^\psi(A) \le \mu_{\Lambda^{(\delta)}}^\xi(A) \le C_\delta \cdot \mu_{\Lambda^{(\delta)}}^\psi(A).
\end{equation}
This result (together with the spatial Markov property \eqref{eq:smp}) has the following two consequences, which are the results that we will use in the present paper. Let $\beta\le \beta_c$.
\begin{itemize}
\item For every $\delta>0$, there exist $c_\delta,C_\delta>0$ such that the following holds for every $n\ge 1$. Let $\Lambda$ be a rhombus such that the thickened rhombus $\Lambda^{(\delta)}$ is included in $\Lambda_n$, let $A \subseteq\Omega_n$ be a static event that depends only on $\Lambda$ and let $B \subseteq\Omega_n$ be a static event that depends only on $\Lambda_n\setminus \Lambda^{(\delta)}$. Then,
\begin{equation}\label{eq:stat_mix_used}
c_\delta \cdot \mu_n(A\cap B) \le \mu_n(A)\mu_n(B) \le C_\delta \cdot \mu_n(A\cap B).
\end{equation}
\item For every $m\ge 1$, $\eta \in \{-1,+1\}^{\Lambda_{2m}(x)}$ and $x\in \T$, we let $T_{-x}\eta = \eta(x+\cdot) \in \Omega_{2m}$. The second consequence of \eqref{eq:stat_mix_initial} is the following:
There exist $c,C>0$ such that the following holds for every $n\ge m \ge 1$. Let $x \in \T$ such that $\Lambda_m(x) \subseteq \Lambda_n$. Then, for every static event $A\subseteq \Omega_{2m}$ that depends only on $\Lambda_{1.5m}$, we have
\begin{equation}\label{eq:stat_mix_box}
c\cdot \mu_{2m}(A) \le \mu_{2n}(T_xA) \le C\cdot \mu_{2m}(A),
\end{equation}
where $T_xA$ is the event $A$ translated by $x$ i.e.\ $T_xA=\{ \eta \in \Omega_{2n} : T_{-x}(\eta_{|\Lambda_{2m}(x)}) \in A \}$.
\end{itemize}
\subsection{Background on Glauber dynamics}\label{sec:back_dyn}
Let $n \ge 1$ and $\beta\in(0,+\infty)$. In all this section, $(\sigma_s)_{s\ge 0}$ is the Glauber dynamics in $\Lambda_n$ (from Definition~\ref{defi:glauber}) at inverse temperature $\beta$ that starts from $\sigma \sim \mu_n=\mu_{n,\beta}$.
\subsubsection{Semi-group and infinitesimal generator}\label{ssec:generator}
In this section, we recall some general properties about the semi-group and the generator of the Glauber dynamics. (See for instance \cite[Chapters 1 and 2]{Lig85} for more details.) The semi-group $(P_t)_{t \ge 0}$ of this dynamics is defined by the following formula. Let $f : \Omega_n \rightarrow \R$. Then for every $t \ge 0$, $P_t f : \Omega_n \rightarrow \R$ is the function
\[
\E\big[f(\sigma_t)\;\big|\;\sigma=\cdot\big].
\]
The infinitesimal generator $L : \R^{\Omega_n} \rightarrow \R^{\Omega_n}$ is defined by
\[
L f=\Big(\frac{d}{dt}P_t f\Big)_{|t=0}.
\]
Let us state a few standard key properties of $P_t$ and $L$.
\begin{remark}\label{rem:gen}
\begin{itemize}[noitemsep]
\item For every $s,t\ge 0$, $P_t\circ P_s=P_{t+s}$;
\item For every $t \ge 0$, $P_t(L f)=L(P_t f)=\frac{d}{dt}P_t f$;
\item If $f:\Omega_n\rightarrow\R$ is increasing, then for any $t\ge 0$, $P_t f$ is increasing.\footnote{This last property comes from the fact that, if one starts the Glauber dynamics with two different initial values $\eta,\eta'\in\Omega_n$ such that $\eta'(x)\ge \eta(x)$ for every $x\in \Lambda_n$ and uses the same Poisson clocks $Y_x$ and uniforms $U_{x,i}$ in both cases, then the monotonicity between the two configuration is conserved in time. To show this last fact, one can use that $\eta \in \Omega_n \mapsto c_x(\eta^{x\leftarrow-})$ is increasing and $\eta \in \Omega_n \mapsto c_x(\eta^{x\leftarrow +})$ is decreasing where, for every $x\in \Lambda_n$ and $\eta\in\Omega_n$, $\eta^{x\leftarrow\pm}$ is the spin configuration that is obtained from $\eta$ by assigning the value $\pm1$ to the spin at~$x$.}
\end{itemize}
\end{remark}

\paragraph{A formula for the infinitesimal generator.}\label{sec:infin-gener} Recall that, given a spin configuration $\eta\in \Omega_n$, we let $\eta^x$ denote the spin configuration obtained from $\eta$ by flipping the spin at $x$. The infinitesimal generator satisfies the following formula (see the appendix for the corresponding computation). For every function $f : \Omega_n \rightarrow \R$ and every $\eta\in \Omega_n$,
\[
L f(\eta)=\sum_{x\in \Lambda_n}c_x(\eta)\big( f(\eta^x)-f(\eta)\big).
\]
\paragraph{Further properties.} We will also rely on the following `standard' results about reversible Markov processes. Let $f : \Omega_n \rightarrow \R$. Then,
\begin{itemize}[noitemsep]
\item For every $t \ge 0$, $\E [f(\sigma)f(\sigma_t)] \ge (\int fd\mu_n)^2$ (indeed, by reversibility in the last equality, we have $\int f d\mu_n = \int P_{t/2}fd\mu_n \le \sqrt{\int (P_{t/2}f)^2d\mu_n} = \sqrt{\E [f(\sigma)f(\sigma_t)]}$);
\item The function $t \mapsto \E [f(\sigma)f(\sigma_t)]$ is non-increasing (see for instance (1.4.2) and the computation above (1.7.2) from \cite[Chapter 1]{BGL13}); 
\item The function $t \mapsto \E [f(\sigma)f(\sigma_t)]$ is convex (see for instance the computations (1.7.2) from \cite[Chapter 1]{BGL13}).
\end{itemize}
\subsubsection{A general differential formula}\label{ssec:diff}
The strategy proposed in \cite{TV23} to prove noise sensitivity properties is based on the study of (integro-)differential inequalities satisfied by $(n,t) \mapsto \pi_n(t)$. As a first step, let us state a general differential formula for $t \mapsto \E [f(\sigma)g(\sigma_t)]$ (that is not new; this can be seen as a particular case of general differential formulas for Markov semi-groups, see e.g.\ \cite{BGL13}). For completeness, we prove this formula in the appendix (Section \ref{sec:app_diff}).
To state the differential formula in question, we introduce a dynamical process $(\sigma_s^{(x)})_{s\ge 0}$ for every $x \in \Lambda_n$. This dynamical process is defined almost exactly like $(\sigma_s)_{s\ge 0}$: it is defined on the same probability space, by using the same Poisson clocks $Y_y$, the same uniforms $U_{y,i}$ and the same rates $c_y$ (see Definition~\ref{defi:glauber}), but by flipping the spin at $x$ at time $0$ (i.e.\ by starting from the configuration $\sigma^x$ instead of $\sigma$).
\paragraph{A general differential formula.} For any $t\ge 0$ and $f,g:\Omega_n \rightarrow \R$, we have
\[
-\frac{d}{dt}\E \big[f(\sigma)g(\sigma_t)\big] = \frac{1}{2}\sum_{x \in \Lambda_n} \E \Big[ c_x(\sigma) \big( f(\sigma^x)-f(\sigma) \big) \big( g(\sigma^{(x)}_t)-g(\sigma_t) \big) \Big].
\]
\subsubsection{Dynamical FKG inequality}\label{ssec:dyn_fkg}
We will need the following dynamical FKG inequality. 
\paragraph{Dynamical FKG.} Let $F,G : (\Omega_n)^2 \rightarrow \R$ be two increasing functions.\footnote{Saying that $F$ is increasing means that both $F(\eta,\cdot)$ and $F(\cdot,\eta)$ are increasing for every $\eta\in \Omega_n$.} Then, for any $t \ge 0$,
\[
\E\big[F(\sigma,\sigma_t)G(\sigma,\sigma_t)\big] \ge \E\big[F(\sigma,\sigma_t)\big]\E\big[G(\sigma,\sigma_t)\big].
\]
\begin{proof}
This property is for instance a consequence of \cite[Corollary~2.21 of Chapter 2]{Lig85}. Indeed, by this corollary (together with \eqref{eq:fkg} and the third item of Remark~\ref{rem:gen}), it is sufficient to show that for every increasing functions $f,g : \Omega_n \rightarrow \R$, we have
\[
L(fg)\ge fLg+gLf.
\]
Let $f$ and $g$ be two such functions and let $\eta\in \Omega_n$. Recall that $\eta^x$ is obtained from $\eta$ by flipping the spin at $x$. By the formula for $L$ stated in Section \ref{ssec:generator}, the function $L( fg) - fLg - gLf$ evaluated at $\eta$ equals
\begin{align*}
&\sum_{x\in \Lambda_n}c_x(\eta) \big\{ f(\eta^x)g(\eta^x) - f(\eta)g(\eta) -f(\eta)g(\eta^x)+f(\eta)g(\eta) - f(\eta^x)g(\eta)+f(\eta)g(\eta)\big\}\\
& = \sum_{x\in \Lambda_n} c_x(\eta) \big(f(\eta^x)-f(\eta)\big)\big(g(\eta^x)-g(\eta)\big),
\end{align*}
which is non-negative since $f(\eta^x)-f(\eta)$ and $g(\eta^x)-g(\eta)$ have the same sign for every $x$.
\end{proof}
\begin{remark}\label{rem:dyn_fkg}
We will actually use the dynamical FKG inequality for a Glauber dynamics with general boundary conditions. Let $V$ be a finite subset of $\T$, let $\xi$ be a boundary condition on $V$ (i.e.\ $\xi \in\{-1,0,+1\}^{\partial_{\mathrm{ext}} V}$) and let $(\sigma^\xi_t)_{t \ge 0}$ be a Glauber dynamics in $V$ with boundary condition $\xi$ that starts from a random spin configuration $\sigma^\xi \sim \mu_V^\xi$ (recall that $\mu_V^\xi$ is the Ising measure on $V$ with boundary condition $\xi$). This dynamics is defined exactly like in Definition \ref{defi:glauber} but by replacing the state space $\Omega_n$ by $\Omega_V=\{-1,+1\}^V$ and the rates $c_x$ by $c_x^\xi(\eta) := \mu_V^\xi(\eta^x)/(\mu_V^\xi(\eta)+\mu_V^\xi(\eta^x))$.

The dynamical FKG inequality holds for this dynamics (and the proof is exactly the same): for any increasing functions $F,G : (\Omega_V)^2 \rightarrow \R$ and any $t \ge 0$,
\[
\E\big[F(\sigma^\xi,\sigma_t^\xi)G(\sigma^\xi,\sigma_t^\xi)\big] \ge \E\big[F(\sigma^\xi,\sigma_t^\xi)\big]\E\big[G(\sigma^\xi,\sigma_t^\xi)\big].
\]
\end{remark}
\subsubsection{Dynamical mixing}\label{ssec:dyn_mix}
\paragraph{The relaxation time is less than a constant if $\beta<\beta_c$.} Assume that $\beta<\beta_c$. Then, there exists $c>0$ such that, for any $n\ge 1$ and any static event $A\subseteq\Omega_n$,
\begin{equation}\label{eq:dyn_mix}
\big|\Pro(\sigma,\sigma_t \in A) -\mu_n(A)^2\big| \le e^{-ct}.
\end{equation}
This is proven in \cite[Theorme 3.2]{MOS94} under the assumption that the model satisfies what is called `Strong Mixing Condition for boxes' (condition (ii) from this theorem). This condition can be proven for any $\beta<\beta_c$ for instance by using the coupling with FK percolation as at the end of Section \ref{sec:app_box}.

\section{Finite-energy for the pair \texorpdfstring{$(\sigma,\sigma_t)$}{} for general \texorpdfstring{$\beta$}{}}\label{sec:ins_tol}
Static Ising model at positive temperature satisfies a finite-energy property: the cost of a local modification  is bounded away from $0$ and $1$  independently of the size $n$ of the system. The main goal of this section is to establish a  finite-energy property in the dynamical setting: at a fixed positive time $t>0$, we will also be able to perform local modifications on the pair $(\sigma,\sigma_t)$. 
\subsection{Definition of the time and finite-energy constants}
\label{ssec:defin-time-insert}
Fix $\beta>0$. In this section we define two constants that will play an important role in the rest of the paper: a time constant $\tau=\tau(\beta)>0$, and a finite-energy  constant $c_{\mathrm{FE}}=c_{\mathrm{FE}}(\beta)>0$.
\paragraph{Definition of the finite-energy constant.}  In this paragraph, we introduce a constant $\cit=\cit(\beta)$  that  will allow us to estimate the costs of local modifications on the pair $(\sigma,\sigma_t)$ (via Proposition~\ref{prop:the_one_we_prove} and Corollary~\ref{cor:ins}).
Let  $a=a(\beta)>0$ be the largest constant  such that the transition rates $c_x$ (as defined in Section~\ref{ssec:setup})  satisfy
\begin{equation}
  \label{eq:3}
  \forall n\ge1 \quad \forall \eta\in \Omega_n \quad \forall x \in \Lambda_n \quad  c_x(\eta)\in [a,1-a].
\end{equation}
We let
\[
\cit=\frac{a^{14}}{16}.
\]
  
\paragraph{Definition of the time constant.} Most of our analysis will be carried out for some  arbitrary but sufficiently small time. For instance, we establish our  differential formula in Section~\ref{sec:mix_gen} and  the spatial mixing property  in Section \ref{sec:mix_sub} when the time satisfies $t\le \tau$. The analysis is much simpler in this perturbative regime because we have a good control on the range of interactions induced by the dynamics. In this paragraph, we define the time constant $\tau$.

\smallskip

Consider site percolation on the triangular lattice and let $\mathcal{C}$ be the union of the cluster of the origin $o \in \T$ and of the clusters of the neighbours of $o$. The number of connected subsets of $\T$ of size $k$ containing $o$ is at most $N^k$ for some $N>0$.\footnote{See for instance \cite[(4.24), Chapter 4]{Gri99} for the existence of such a constant $N$ in the case of $\Z^2$ -- the proof for the triangular lattice is exactly the same.}. As a result, we can (and we do) choose some absolute constant  $M>0$ such that
\[
\forall k \ge 0 \quad \big|\{C\subset \T : |C|=k \text{ and $C$ is a possible outcome for $\mathcal{C}$}\}\big|\le M^k.
\]
We then define
\begin{equation}
  \label{eq:7}
  \tau=\frac{(\cit)^2}{10^6 M}.
\end{equation}
\begin{remark}
This choice implies  in particular the following estimates (that we will use later and which the reader does not need to interpret right now):
\begin{equation}
    \label{eq:6}
    \sum_{k\ge 0}  \bigg( \frac{1}{2\sqrt{\cit}} \bigg)^{k+1}M^k \tau^k  \le \frac{1}{\sqrt{\cit}},
  \end{equation}
  \begin{equation}
    \label{eq:16}
    \forall \lambda\ge 1 \quad \sum_{k\ge \lambda} \bigg(\frac{16}{\cit}\bigg)^{k+1}M^k\tau^k\le e^{-\lambda},
  \end{equation}
and
\begin{equation}
  \label{eq:18}
  \sum_{k \ge 0} (k+1)\bigg(\frac{16}{\cit}\bigg)^{k+1}M^k \tau^k (2k+1)^2 \le \frac{100}{\cit}.
\end{equation}
\end{remark}

\subsection{Dynamical finite-energy property}
\label{sec:finite-energy-one}

Fix $\beta>0$, and let $\tau$ and $\cit$ be as defined in the previous section.  In all this section, we also fix some $n\ge 1$ and consider the Glauber dynamics $(\sigma_s)_{s \ge 0}$  in $\Lambda_n$ (from Definition \ref{defi:glauber}) starting from $\sigma\sim\mu_n=\mu_{n,\beta}$.
We are interested in the finite-energy properties of the pair $(\sigma,\sigma_t)$, seen  as a random  element of the hypercube  $\{-1,+1\}^{\Lambda_n}\times\{-1,+1\}^{\Lambda_n}$.  We refer to $\sigma$ and $\sigma_t$ as the first and second component respectively. A difficulty is that the finite-energy cost depends on the time~$t$: if $t=0$, we must have $\sigma_t(x)=\sigma(x)$ and therefore, we cannot only modify one component  at a site. If $t>0$, changing a configuration  where $\sigma_t(x)=\sigma(x)$  to a configuration where $\sigma_t(x)\neq\sigma(x)$ costs at least $t$ because the clock at $x$ must ring at least once. Our first result (Proposition~\ref{prop:the_one_we_prove}) states that $t$ is, up to constants, the right cost to pay for such modification. In our applications, we will use the finite-energy simultaneously on both components (we allocate the same spin value at a site at time $0$ and at time $t$). In this case, the finite-energy cost is independent of $t$, see Corollary~\ref{cor:ins}.  
For $\eta \in \Omega_n$ and $x\in \Lambda_n$, recall that   $\eta^x$  denotes the configuration  obtained from $\eta$ by flipping the spin at~$x$.

\begin{proposition}\label{prop:the_one_we_prove}
For every $t \in [0,\tau]$, $x\in\Lambda_n$, and $\eta,\psi \in \Omega_n$, we have
\[
\Pro(\sigma=\eta,\sigma_t=\psi^x)  \ge \sqrt{\cit}  \cdot t^\delta\cdot \Pro(\sigma=\eta,\sigma_t=\psi),
\]
where  $\delta=\psi(x)\eta(x) \in \{-1,1\}$.
\end{proposition}

\begin{remark}
  By reversibility, we can also modify the first coordinate. With the same notation as above we also have
  \begin{equation}
    \label{eq:10}
    \Pro(\sigma=\eta^x,\sigma_t=\psi)  \ge \sqrt{\cit}  \cdot t^\delta\cdot \Pro(\sigma=\eta,\sigma_t=\psi).
  \end{equation}  
\end{remark}
\begin{remark}
Proposition \ref{prop:the_one_we_prove} is stated for $t \in [0,\tau]$, and this is what we will use in the present paper; but let us observe that the proposition implies a finite-energy property at any time. Indeed, by conditioning on $\sigma_{t-\tau}$ when $t \ge \tau$, we obtain the following result (where $\tau$ appears because it equals $\tau \wedge (1/\tau)$): Let $\eta,\psi \in \Omega_n$. Then, for any $t \ge \tau$, any $x\in\Lambda_n$ and any $\eta,\psi \in \Omega_n$,
\[
\Pro(\sigma=\eta,\sigma_t=\psi^x) \ge \sqrt{\cit} \cdot \tau \cdot \Pro( \sigma=\eta,\sigma_t=\psi ).
\]
\end{remark}

\paragraph{A general sketch of proof.} Before writing the proof of Proposition \ref{prop:the_one_we_prove}, let us write a sketch of proof for this result, that will also be useful in order to read the proofs of the other `general' properties of the pair $(\sigma,\sigma_t)$ proven in Sections~\ref{sec:ins_tol}--\ref{sec:mix_sub}. As explained in Section \ref{ssec:high}, the main difficulty is that, contrary to the static configuration $\sigma$, the pair $(\sigma,\sigma_t)$ does not satisfy the spatial Markov property.\footnote{Let us be a little more precise: the spatial Markov property for the pair would be the property that for any $V\subseteq\Lambda_n$, the law of $(\sigma,\sigma_t)_{|V}$ conditionally on $(\sigma,\sigma_t)_{|\Lambda_n\setminus V}$ is the same as if we just condition on $(\sigma,\sigma_t)$ restricted to $\Lambda_n\cap \partial_\mathrm{ext}V=\{x \in \Lambda_n\setminus V : \exists y \in V, y \sim x \}$. To build an image, the reader can for instance study the graph $G$ made of three vertices $\{x,y,z\}$ and of two edges $\{x,y\}$ and $\{y,z\}$, and to shorten the calculation they can consider the case $t$ larger than $0$ but extremly small, and $\beta$ very large, and prove for instance that 
\begin{multline*}
\Pro\big((\sigma(x),\sigma_t(x))=(+1,+1) \; \big| \; (\sigma(y),\sigma_t(y))=(+1,-1),(\sigma(z),\sigma_t(z))=(-1,-1)\big)\\>\Pro\big((\sigma(x),\sigma_t(x))=(+1,+1) \; \big| \; (\sigma(y),\sigma_t(y))=(+1,-1),(\sigma(z),\sigma_t(z))=(+1,+1)\big).
\end{multline*}
}

In order to overcome this, given some site $x$, we will define a random set containing $x$ such that some relevant information (for instance, the fact that the clocks have rung or not) in the interior of this set is independent of the exterior of it. This set will be the `red cluster' of $x$ (and of its neighbours) in the proof of Proposition~\ref{prop:the_one_we_prove}. In Sections \ref{sec:mix_gen} and \ref{sec:mix_sub}, it will be the `green cluster' of $x$ (and of its neighbours). This independence property can be seen as a weak spatial Markov property.

As an example of such an independence property, the reader can have in mind that if we fix a circuit $\gamma$ and we condition on $\sigma_{|\gamma}$ and on the the event that the clocks in $\gamma$ have not rung before some fixed time $t$, then (if we consider only the dynamics until time $t$) the dynamical process in the region surrounded by $\gamma$ is independent of the dynamical process outside this region.

In Sections \ref{sec:ins_tol} and \ref{sec:mix_gen}, we essentially prove that these random `red' and `green' clusters are typically small. Roughly speaking, this means that the information of the value of the spin at $x$ travels very slowly, which helps to prove finite-energy and spatial mixing properties. The finite-energy property of the static Ising model will be the key tool to prove that the `red' clusters are small (or more precisely that the cost of making them large is always greater than the cost of assigning the desired values to $x$, see \eqref{eq:red_is_small} below), which will enable us to prove the finite-energy property of the pair $(\sigma,\sigma_t)$, i.e.\ Proposition~\ref{prop:the_one_we_prove}. Then, the finite-energy property of the pair will be they key tool to prove that the `green' clusters are small, see Sections \ref{ssec:CS} and \ref{sec:4.2} (in particular Lemmas \ref{lem:mix_dec} and~\ref{lem:3}). Once we know that the green clusters are small, we can prove several properties of the pair $(\sigma,\sigma_t)$, see Section \ref{sec:diff} and Section \ref{sec:mix_sub}. In particular, in Section \ref{ssec:mix}, we use this in order to transfer the spatial mixing property of the static Ising model to the pair $(\sigma,\sigma_t)$ (when $\beta\le\beta_c$).
\begin{proof}[Proof of Proposition \ref{prop:the_one_we_prove}] Fix two spin configurations $\eta,\psi \in \Omega_n$, some $x \in \Lambda_n$ and a time $t\in [0,\tau]$. Moreover, let $D=\{z\in \Lambda_n : \eta(z) \neq \psi(z) \}$ denote the (deterministic) disagreement set of $\eta$ and $\psi$, and $D'=D\Delta\{x\}$ be the disagreement set of $\eta$ and $\psi^x$.

  \smallskip
  
  We call red sites all sites in $\Lambda_n$ whose clock has rung more than necessary for the event $\{\sigma=\eta,\sigma_t=\psi\}$ to hold. More precisely, any $z \in \Lambda_n\setminus D$ is red if its clock has rung at least once between times $0$ and $t$ and any $z \in D$ is red if its clock has rung at least twice between times $0$ and $t$. We let $\mathcal{R}$ be the set of red sites  and  $\mathcal{R}_x$ denote the union of the red component of $x$ and of the red components of the neighbours of~$x$ (Remark: By convention, if a site is not red then its red component is the empty set. In particular, $\mathcal{R}_x=\emptyset$ means that $x$ is not red and has no red neighbour).  Our goal is to prove the following statement:
  \begin{equation}\label{eq:red_is_small}
    \forall R\subseteq \Lambda_n \quad  \mathbb P(\sigma=\eta,\sigma_t=\psi,\mathcal{R}_x=R)\le t^{|R|-\delta} \left (\frac 2 {a^7}\right)^{|R|+1} \mathbb P(\sigma=\eta,\sigma_t=\psi^x),\end{equation}
 where $a$ is the constant from Equation~\eqref{eq:3}. 
  The proposition follows by summing above all admissible $R$, by using that the number of such sets with cardinality $k$ is at most $M^k$ (recall that $M$ is an absolute constant that has been fixed in Section~\ref{ssec:defin-time-insert}) and by using Equation \eqref{eq:6}. 
  
  \smallskip

We work under the measure $\mathbb P_\eta:=\mathbb P(\ \cdot\ | \:\sigma=\eta)$. In particular, since the original configuration is fixed, all the considered events only depend on the Poisson clocks $Y_z$, $z\in \Lambda_n$ and the uniforms $U_{z,i}$, $z\in \Lambda_n$, $i\ge 1$. 
  With this notation, and observing that $\delta=|D'|-|D|$, our goal is to show that
  \begin{equation}
        \forall R\subseteq \Lambda_n \quad  \mathbb P_\eta(\sigma_t=\psi,\mathcal{R}_x=R)\le t^{|D|-|D'|+|R|}  \left (\frac 2 {a^7}\right)^{|R|+1} \mathbb P_\eta(\sigma_t=\psi^x). \label{eq:15}
  \end{equation}
  Fix some admissible $R\subseteq \Lambda_n$ and write $S=R\cup  \{x\}$ (in particular $S=R$ if $x\in R$). Recall that $\partial_{\mathrm{ext}}S=\{y \in \T\setminus S : \exists x \in S, x \sim y \}$. In this proof, we will intersect $\partial_{\mathrm{ext}}S$ with $\mathcal{R}$, which is a subset of $\Lambda_n$, so when the reader thinks of $\partial_{\mathrm{ext}}S$, they may have $\{y \in \Lambda_n\setminus S : \exists x \in S, x \sim y \}$ in mind. First observe that 
  \begin{equation}
      \{\mathcal R_x=R\} \subseteq \{R\subseteq \mathcal R, \partial_{\mathrm{ext}} S\cap \mathcal R=\emptyset\}.
  \end{equation}
  Hence, if we define 
\[
    \mathcal{A}=\begin{cases}\{R\subseteq \mathcal R\} & \text{ if $x \notin D$}\\
    \{ R \subseteq \mathcal{R}, \text{ the clock of $x$ has rung before time $t$}\} & \text{ if $x \in D$}
    \end{cases}
    \]
    and
    \[ \mathcal{B}=\{\sigma_t|_{S^c}=\psi|_{S^c},\, \partial_{\mathrm{ext}} S\cap \mathcal R=\emptyset\},
\]
we have
\begin{equation}\label{eq:14}
    \mathbb P_\eta(\sigma_t=\psi,\mathcal{R}_x=R)\le \mathbb P_\eta (\mathcal{A}\cap \mathcal{B}).
  \end{equation}
  The events $\mathcal{A}$ and $\mathcal{B}$ are not independent in general,  and we wish to `decorrelate' them. To do this, we  consider the event $\mathcal{B}'=\mathcal{B}\cap \{ \forall z\in\partial_{\mathrm{ext}} S\cap D, U_{z,1}\ge 1-a\}$ where the uniforms at the exterior boundary of $S$ are chosen in such a way that the spin value can be updated without looking at the values of  the neighbours.  In order to replace the event $\mathcal{B}$ by $\mathcal{B}'$, we use FKG inequality. Let $\mathcal G$   the sigma-field generated by all the clocks $Y_z$ for $z\in \Lambda_n$ and the uniforms $U_{z,i}$, $z\notin \partial_{\mathrm{ext}} S\cap D$, $i\ge 1$.   Conditionally on $\mathcal G$, the event $\mathcal{A}\cap \mathcal{B}$ is increasing in the uniforms  $U_{z,1}$,    $z\in \partial_{\mathrm{ext}} S\cap D$. Hence, by the FKG inequality and using that the uniforms in $\partial_{\mathrm{ext}} S\cap D$ are independent of $\mathcal G$, we get 
  \begin{equation}
      \mathbb P_\eta(\mathcal{A}\cap \mathcal{B}'|\mathcal G) \ge a^{|\partial_{\mathrm{ext}} S\cap D|} \cdot \mathbb P_\eta(\mathcal{A}\cap \mathcal{B}|\mathcal G ).
    \end{equation}
An observation that compensates for the lack of spatial Markov property for the pair $(\sigma,\sigma_t)$ is that the events $\mathcal{A}$ and $\mathcal{B}'$ are independent because $\mathcal{B}'$ is measurable with respect to the clocks and uniforms in $S^c$. Hence taking the expectation in the equation above and using $|\partial_{\mathrm{ext}} S\cap D|\le 6|S|$, we get
    \begin{align}
      \label{eq:12}
      \mathbb  P_\eta(\mathcal{A}\cap \mathcal{B})& \le a^{-6|S|} \cdot \mathbb P_\eta(\mathcal{A})\mathbb P_\eta(\mathcal{B}')\nonumber\\
      & = a^{-6|S|} \cdot (1-e^{-t})^{|R \setminus D|+1_{x\in D\setminus R}}\cdot (1-e^{-t}-te^{-t})^{|D\cap R|} \cdot \mathbb P_\eta(\mathcal{B}')\nonumber\\
      & \le\frac{t^{|D\cap S|+|R|}}{a^{6|S|}}\mathbb P_\eta(\mathcal{B}').
    \end{align}
Recall that $D'=D\Delta\{x\}$ and let $\mathcal{A}'$ be the event that
    \begin{itemize}[noitemsep]
    \item all the clocks in $D'\cap S$ ring exactly once,
    \item no  clock in $S\setminus D'$ rings, and
    \item all the first uniforms for $z\in S$ satisfy $U_{z,1}\ge 1-a$.
    \end{itemize}
    A clock at a fixed site rings exactly once with probability  $te^{-t}\ge t/2$, and it does not ring with probability  $1-e^{-t}\ge 1/2$. Hence, the event $\mathcal{A}'$ occurs with probability at least $t^{|D'\cap S|}(a/2)^{|S|}$. 
    Since $\mathcal{A}'$ depends only on the clocks and uniforms in $S$, it is independent of $\mathcal{B}'$. Therefore,
    \begin{equation}
      \label{eq:13}
      t^{|D'\cap S|}(a/2)^{|S|}\mathbb P_\eta(\mathcal{B}') \le \mathbb P_\eta(\mathcal{A}'\cap \mathcal{B}')\le \mathbb  P_\eta(\sigma_t=\psi^x),
    \end{equation}
    where the last inequality comes from the fact that on $\mathcal{A}'$, no spin in $S\setminus D'$ has changed and all sites of $S\cap D'$ have changed.
    
    Together with \eqref{eq:14} and \eqref{eq:12}, the displayed equation above concludes the proof of the desired equation~\eqref{eq:15}.
\end{proof}
We now discuss the cost of a  `synchronized' finite modification, where we allocate the same spin value at time $0$ and at time $t$. For $\eta \in \Omega_n$, $x\in\Lambda_n$ and $s \in\{-1,+1\}$, let  $\eta^{x\leftarrow s}$ be  the spin configuration obtained from $\eta$ by assigning the value $s$ to the spin at $x$.
\begin{corollary}\label{cor:ins}
For any $t \in [0,\tau]$, $x\in\Lambda_n$,  $\eta,\psi \in \Omega_n$ and  $s \in \{-1,+1\}$, we have
\[
\Pro(\sigma=\eta^{x\leftarrow s},\sigma_t = \psi^{x\leftarrow s})\ge \cit \cdot \Pro(\sigma=\eta,\sigma_t = \psi).
\]
\end{corollary}
\begin{proof}
  The statement is trivial if $\eta(x)=\psi(x)=s$. If $\eta(x)\neq \psi(x)$, by reversibility, we can assume that $\eta(x)=s$ and $\psi(x)=-s$. In this case, we have
  \begin{align}
     \Pro(\sigma=\eta^{x\leftarrow s},\sigma_t = \psi^{x\leftarrow s})  &\overset{\phantom{\text{Prop.~\ref{prop:the_one_we_prove}}}} =  \Pro(\sigma=\eta,\sigma_t = \psi^{x})
    \\ &\overset{\text{Prop.~\ref{prop:the_one_we_prove}}}{\ge} \frac{\sqrt{\cit}} t  \cdot  \Pro(\sigma=\eta,\sigma_t = \psi) \ge \cit\cdot \Pro(\sigma=\eta,\sigma_t = \psi). 
  \end{align}
Finally, if $\eta(x)=\psi(x)=-s$, we change one coordinate after the other and apply Proposition \ref{prop:the_one_we_prove} twice (or more precisely \eqref{eq:10} and Proposition \ref{prop:the_one_we_prove}). We have  
  \begin{align}
    \Pro(\sigma=\eta^{x\leftarrow s},\sigma_t = \psi^{x\leftarrow s})  &\overset{\phantom{\text{Prop.~\ref{prop:the_one_we_prove}}}} =  \Pro(\sigma=\eta^x,\sigma_t = \psi^{x})\\
                                                                       &\overset{\ \text{Eq.~\eqref{eq:10}}\ }{\ge} \sqrt{\cit} \cdot  t  \cdot  \Pro(\sigma=\eta,\sigma_t = \psi^x) \\
    &\overset{\text{Prop.~\ref{prop:the_one_we_prove}}}{\ge}   \cit\cdot \Pro(\sigma=\eta,\sigma_t = \psi). \qedhere
  \end{align}
\end{proof}
\section{The events \texorpdfstring{$\{\sigma,\sigma_t \in A\}$}{} for general  \texorpdfstring{$\beta$}{beta}}\label{sec:mix_gen}
In this section, we start our analysis of the events $\{\sigma,\sigma_t \in A\}$ by using the finite-energy property of the previous section. In particular, we prove that we can intersect these events with a `decoupling event' without changing much its probability (see Section \ref{sec:4.2}) and we prove a differential inequality for $t \mapsto \Pro( \sigma,\sigma_t \in A)$ in terms of pivotal events (see Section \ref{sec:diff}).

\smallskip

Fix $\beta >0$. Let $\cit$ and $\tau$ be the constants as defined in Section~\ref{ssec:defin-time-insert}. In all this section, we also fix some $n\ge1$ and consider the Glauber dynamics $(\sigma_s)_{s\ge 0}$ in $\Lambda_n$ (from Definition~\ref{defi:glauber}) that starts from $\sigma \sim \mu_n=\mu_{n,\beta}$.

Recall that in this paper a static event is just a subset of $\Omega_n=\{-1,1\}^{\Lambda_n}$.

\subsection{A Cauchy--Schwarz trick}\label{ssec:CS}
We start with a rather elementary `trick' that we will use both in the proof of `general' results about the events $\{\sigma,\sigma_t \in A\}$ (in Sections \ref{sec:mix_gen} and \ref{sec:mix_sub}) and in the percolation sections of the paper.
\begin{lemma}\label{lem:CS}
Let $A,B\subseteq\Omega_n$. For every $t\ge 0$,
\[
\Pro(\sigma \in A,\sigma_t \in B)\le \max\big\{\Pro(\sigma,\sigma_t \in A),\Pro(\sigma,\sigma_t \in B)\big\}.
\]
\end{lemma}
\begin{proof}
By the Markov property, the Cauchy--Schwarz inequality and then reversibility, we have
\begin{align*}
\Pro(\sigma \in A,\sigma_t \in B) & = \E \Big[ \Pro\big[ \sigma \in A \; \big| \; \sigma_{t/2} \big]  \Pro\big[ \sigma_t \in B \; \big| \; \sigma_{t/2} \big] \Big]\\
& \le \sqrt{\E \Big[ \Pro\big[ \sigma \in A \; \big| \; \sigma_{t/2} \big]^2 \Big]\E \Big[ \Pro\big[ \sigma_t \in B \; \big| \; \sigma_{t/2} \big]^2 \Big]}\\
& = \sqrt{\Pro(\sigma,\sigma_t \in A) \Pro(\sigma,\sigma_t \in B)}. \qedhere
\end{align*}
\end{proof}
Let us now state and prove a lemma that is a consequence of the finite-energy property and the Cauchy--Schwarz trick.
\paragraph{Definition of the green set.} 
 Recall that the Glauber dynamics $(\sigma_s)_{s\ge 0}$ is defined by sampling an initial configuration $\sigma$ and then updating it by using the Poisson clocks $Y_x$ and the uniform variables $U_{x,i}$ (see Definition \ref{defi:glauber}). We call green sites all sites in $\Lambda_n$ whose clock has rung before time $\tau$, we let $\mathcal{G}$ be the set of green sites and, for every $x \in \Lambda_n$ and we let  $\mathcal{G}_x$ denote the union of the green component of $x$ and of the green components of the neighbours of~$x$ (Remark: By convention, if a site is not green then its green component is the empty set. In particular, $\mathcal{G}_x=\emptyset$ means that $x$ is not green and has no green neighbour). 
 If $H\subseteq\Lambda_n$, $\eta \in \Omega_n$ and $\alpha\in \{-1,+1\}^H$, we let $\alpha|\eta$ be the spin configuration that coincides with $\alpha$ on $H$ and with $\eta$ in $\Lambda_n\setminus H$. Moreover, we let
\begin{equation}
  \label{eq:11}
  A_{\alpha}=\{\eta\in \Omega_n\: : \: \alpha|\eta\in A\}.  
\end{equation}
(If the reader wants motivation before reading the proof of the following lemma, they can first read Section \ref{sec:4.2}.)
\begin{lemma}\label{lem:ins_tol}
Let $x\in \Lambda_n$, $G \subseteq \Lambda_n$ and $A,B\subseteq \Omega_n$. Let $H=G\cup\{x\}$. Then, for every $t \in [0,\tau]$,
\begin{multline*}
\sum_{\alpha,\beta\in\{-1,+1\}^H} \Pro\big( \sigma\in A_\alpha,\sigma_t \in B_\beta,\mathcal{G}_x=G \big)\\
\le \bigg(\frac{16}{\cit}\bigg)^{|G|+1} \cdot \tau^{|G|} \cdot \max\big\{\Pro\big( \sigma,\sigma_t \in A \big),\Pro\big( \sigma,\sigma_t \in B \big)\big\}.
\end{multline*}
\end{lemma}
\begin{proof}
Let $x,G,A,B,H,t$ as in the statement and let $\alpha,\beta\in\{-1,+1\}^H$. Recall that $\partial_{\mathrm{ext}}H=\{y \in \T\setminus H : \exists x \in H, x \sim y \}$. In this proof, we will intersect $\partial_{\mathrm{ext}}H$ with $\mathcal{G}$, which is a subset of $\Lambda_n$, so when the reader thinks of $\partial_{\mathrm{ext}}H$, they may have $\{y \in \Lambda_n\setminus H : \exists x \in H, x \sim y \}$ in mind. 

We note that the event $\{\mathcal{G}_x=G\}$ is included in the event $\{ G\subseteq \mathcal{G}, \partial_{\mathrm{ext}} H\cap \mathcal G=\emptyset\}$.
An observation concerning the green set that compensates for the lack of spatial Markov property for the pair $(\sigma,\sigma_t)$ is that the event $\{\partial_{\mathrm{ext}} H\cap \mathcal G=\emptyset,\sigma\in A_\alpha,\sigma_t\in B_\beta\}$ is measurable with respect to the initial configuration $\sigma$ and all the Poisson clocks and uniforms outside $H$ (because in this event the spins at the exterior boundary of $H$ have not been updated). In particular, it is independent of the Poisson clocks in $H$. This independence property gives
\begin{align}
      \Pro\big(\sigma\in A_\alpha,\sigma_t \in B_\beta, \mathcal{G}_x=G\big) &\le \tau^{|G|}    \mathbb P\big(\partial_{\mathrm{ext}} G\cap \mathcal G=\emptyset,\sigma\in A_\alpha,\sigma_t\in B_\beta\big)\\
  &\le  \tau^{|G|}  \mathbb P\big(\sigma \in A_\alpha,\sigma_t\in B_\beta\big)\\
        &\le \tau^{|G|}   \max \big\{ \mathbb P(\sigma,\sigma_t\in A_\alpha) , \mathbb P(\sigma,\sigma_t\in B_\beta) \big\} \text{ by Lem.\ \ref{lem:CS}}.
\end{align}
The following claim implies the desired result by summing on every $\alpha,\beta$.
\begin{claim}
We have
\[
\Pro(\sigma,\sigma_t \in A_\alpha) \le  \bigg(\frac{4}{\cit}\bigg)^{|H|}\cdot \Pro(\sigma,\sigma_t \in A).
\]
\end{claim} 
\begin{proof}[Proof of the claim]
By induction, it is sufficient to prove the result when $H$ is a singleton, so let $y\in \Lambda_n$, let $s\in\{-1,+1\}$, let  $\eta^{y\leftarrow s}$ be  the spin configuration obtained from $\eta$ by assigning the value $s$ to the spin at $y$ and let $A_s=\{\eta \in \Omega_n : \eta^{y \leftarrow s} \in A\}$. Our goal is to prove that $\Pro(\sigma,\sigma_t \in A_s) \le (4/\cit)\cdot\Pro(\sigma,\sigma_t \in A)$. By Corollary \ref{cor:ins}, we have
\begin{align*}
\Pro(\sigma,\sigma_t \in A_s) = \sum_{\eta,\psi\in A_s} \Pro(\sigma=\eta,\sigma_t = \psi) & \le \frac{1}{\cit}\sum_{\eta,\psi\in A_s} \Pro(\sigma=\eta^{y \leftarrow s},\sigma_t = \psi^{y \leftarrow s})\\
& \le \frac{4}{\cit} \sum_{\eta,\psi\in A : \eta(y)=\psi(y)=s} \Pro(\sigma=\eta,\sigma_t = \psi)\\
& \le \frac{4}{\cit} \sum_{\eta,\psi\in A} \Pro(\sigma=\eta,\sigma_t = \psi)\\
& = \frac{4}{\cit} \Pro(\sigma,\sigma_t \in A),
\end{align*}
where the factor $4$ in the second inequality comes from the fact that $(\eta,\psi) \mapsto (\eta^{y \leftarrow s},\psi^{y \leftarrow s})$ is four-to-one.
\end{proof}
This ends the proof of the lemma.
\end{proof}
\subsection{The decoupling event}\label{sec:4.2}
In this section, we prove a decoupling result that will imply spatial mixing when we restrict ourself to $\beta\le \beta_c$ in the next section. To state this result, we need a definition. Recall the notation $\mathcal{G}_x$ from Section \ref{ssec:CS}.
\begin{definition}\label{defi:dec}
Given a non-empty set $S \subseteq \Lambda_n$, we define the decoupling event $\mathsf{dec}(S)$ as follows:
\[
\mathsf{dec}(S) = \bigcap_{x \in S} \big\{|\mathcal{G}_x| \le \log|S| \big\}.
\]
\end{definition}
The following lemma enables one to include the decoupling event $\mathsf{dec}(S)$ in the probability $\Pro(\sigma,\sigma_t \in A )$ without changing much this quantity.
\begin{lemma}\label{lem:mix_dec}
For every static event $A \subseteq\Omega_n$, every non-empty set $S \subseteq \Lambda_n$ and every $t\in[0,\tau]$,
\[
\Pro\big( \{\sigma,\sigma_t \in A \} \cap \mathsf{dec}(S)\big) \ge \frac{1}{2}\Pro(\sigma,\sigma_t \in A ).
\]
\end{lemma}
\begin{proof}
This is a direct consequence of a union bound on the points of $S$ and of the following lemma (applied to $\lambda = \log|S|+1$), that implies that
\[
\Pro\big( \{\sigma,\sigma_t \in A \} \cap \mathsf{dec}(S)^c\big) \le |S|e^{-(\log|S|+1)}\Pro(\sigma,\sigma_t \in A ) \le \frac{1}{e}\Pro(\sigma,\sigma_t \in A ). \qedhere
\]
\end{proof}
\begin{lemma}
  \label{lem:3}
For every $A\subseteq \Omega_n$, $x\in \Lambda_n$, $\lambda\ge 1$ and $t \in [0,\tau]$,
  \begin{equation}
    \label{eq:17}
    \mathbb P\big(\sigma,\sigma_t\in A,|\mathcal G_x|\ge \lambda\big)\le e^{-\lambda}\mathbb P\big(\sigma,\sigma_t\in A\big).
  \end{equation}
\end{lemma}
\begin{proof}
Let $A \subseteq \Omega_n$,  $x \in \Lambda_n$, $\lambda\ge 1$ and $t \in [0,\tau]$. By decomposing over all possible values for $\mathcal G_x$, we have
  \begin{equation}
    \label{eq:2}
 \Pro\big(\sigma,\sigma_t\in A,\, |\mathcal G_x| \ge \lambda\big) \le   \sum_{\substack{|G| \ge\lambda}  }\underbrace{\Pro\big(\sigma,\sigma_t\in A, \mathcal G_x=G\big)}_{=:p_{G}} . 
\end{equation}
We  now fix an admissible set $G$ (i.e.\ a set that is a possible outcome for $\mathcal{G}_x$), we let $H=G\cup\{x\}$ and we upper bound the summand $p_{G}$  by summing over the possible configurations in $H$ at time $0$ and at time $t$. We obtain that
\begin{equation}
  \label{eq:9}
  p_{G} \le \sum_{\alpha,\beta \in \{-1,+1\}^H}  \mathbb P( \sigma\in A_\alpha,\sigma_t\in A_\beta, \mathcal{G}_x=G) \le \bigg(\frac{16}{\cit}\bigg)^{|G|+1}\tau^{|G|}\Pro\big( \sigma,\sigma_t \in A \big),
\end{equation}
where the last inequality follows from Lemma \ref{lem:ins_tol} (with $A=B$).
         
Let us now return to \eqref{eq:2}. By first using our bound on $p_G$ and then the definition of $M$ and $\tau$ (see Section~\ref{ssec:defin-time-insert}, Equation~\eqref{eq:16}), we get
\[
\Pro(\sigma,\sigma_t\in A,\, |\mathcal G_x| \ge \lambda) \le \mathbb P(\sigma,\sigma_t\in A) \sum_{k\ge \lambda} \bigg(\frac{16}{\cit}\bigg)^{k+1}M^k\tau^k\le e^{-\lambda}\mathbb P(\sigma,\sigma_t\in A). \qedhere
\]
\end{proof}
\subsection{Differentiation}\label{sec:diff}
In this section, we prove that the interpretation of the differential formulas in terms of pivotal events -- that is true in the case of Bernoulli percolation under i.i.d.\ dynamics (see e.g.\ \cite[Prop.\ 4.2]{TV23}) -- still holds in the case of the Ising model (up to a multiplicative constant and at times $t \le \tau$).
\begin{definition}
Let $W\subseteq\Lambda_n$, $A \subseteq \Omega_n$ and $\eta\in\Omega_n$. We say that $W$ is pivotal for $A \subseteq \Omega_n$ in $\eta$ if there exists $\psi\in\Omega_n$ such that $\eta$ and $\psi$ coincide outside $W$ and $1_A(\psi)\neq 1_A(\eta)$. If $\{x\}$ is pivotal for $A$ in $\eta$, we simply say that $x$ is pivotal for $A$ in $\eta$. We let $\mathsf{Piv}_x(A)\subseteq \Omega_n$ denote the event that $x$ is pivotal for $A$.
\end{definition}
\begin{proposition}\label{prop:diff}
There exists $C>0$ that does not depend on $n$ such that the following holds for every $A\subseteq \Omega_n$ and $t\in[0,\tau]$:
\[
0\le -\frac{d}{dt}\Pro(\sigma,\sigma_t \in A ) \le C \sum_{x \in \Lambda_n} \Pro \big( \sigma,\sigma_t \in \mathsf{Piv}_x(A) \big).
\]
If additionally the event $A$ is increasing, then
\[
\frac{1}{C} \sum_{x \in \Lambda_n} \Pro \big( \sigma,\sigma_t \in \mathsf{Piv}_x(A) \big) \le -\frac{d}{dt}\Pro(\sigma,\sigma_t \in A ) \le C\sum_{x \in \Lambda_n} \Pro \big( \sigma,\sigma_t \in \mathsf{Piv}_x(A) \big).
\]
\end{proposition}
\begin{proof}
Let us first consider the case where $A$ is not necessarily increasing. The fact that $t \mapsto \Pro(\sigma,\sigma_t \in A)$ is non-increasing is a general property of the dynamics (see the end of Section~\ref{sec:infin-gener}). Let us prove the upper bound. We start the proof by using the general differential formula stated in Section~\ref{ssec:diff}. Recall that $(\sigma^{(x)}_s)_{s \ge 0}$ is the same process as $(\sigma_s)_{s \ge 0}$ (and defined with the same underlying Poisson clocks $Y_y$ and uniforms $U_{y,i}$), except that we have changed the value of the spin at $x$ at time $0$ (i.e.\ the initial configuration is $\sigma^x$). Recall also the definition of $\mathcal{G}_x$ in Section~\ref{ssec:CS}. Let $f=1_A$. By the formula from Section~\ref{ssec:diff} (applied to $g=f$), we have
\[
0 \le -\frac{d}{dt}\Pro(\sigma,\sigma_t \in A) = \frac{1}{2}\sum_{x \in \Lambda_n} \E \Big[ c_x(\sigma) \big( f(\sigma^x)-f(\sigma) \big) \big( f(\sigma^{(x)}_t)-f(\sigma_t) \big) \Big].
\]
Since $c_x\le 1$ and since $\sigma^{(x)}_t$ and $\sigma_t$ coincide outside $\mathcal{G}_x\cup\{x\}$, we obtain that
\begin{align*}
& -\frac{d}{dt}\Pro(\sigma,\sigma_t \in A)\\
& \le \frac{1}{2} \sum_{x \in \Lambda_n} \sum_{G } \Pro ( \text{$x$ is piv.\ for $A$ in $\sigma$, $G\cup\{x\}$ is piv.\ for $A$ in $\sigma_t$, $\mathcal{G}_x=G$} ).
\end{align*}
Recall that if $\alpha\in\{-1,+1\}^{G\cup\{x\}}$ and $\eta\in\Omega_n$, we let $\alpha|\eta$ denote the spin configuration that coincides with $\alpha$ in $G\cup\{x\}$ and with $\eta$ in $\Lambda_n\setminus (G\cup\{x\})$. We observe that if the subset $G\cup\{x\}$ is pivotal for $A$ in $\sigma_t$, then there exists $y\in G\cup\{x\}$ and $\alpha\in\{-1,+1\}^{G\cup\{x\}}$ such that $\alpha|\sigma_t \in \mathsf{Piv}_y(A)$. (Indeed, there exist $\beta,\gamma\in\{-1,+1\}^{G\cup\{x\}}$ such that $1_A(\beta|\sigma_t)\neq 1_A(\gamma|\sigma_t)$, and one can find $y$ and $\alpha$ by interpolating between $\beta|\sigma_t$ and $\gamma|\sigma_t$: change the spins in $G\cup\{x\}$ one by one and stop when the value of $1_A$ changes.) As a result,
\begin{align*}
& -\frac{d}{dt}\Pro(\sigma,\sigma_t \in A)\\
& \le \frac{1}{2} \sum_{x \in \Lambda_n} \sum_G  \sum_{y \in G\cup\{x\}} \underbrace{\sum_{\alpha,\beta \in \{-1,+1\}^{G\cup\{x\}}} \Pro \big( \sigma \in \{\mathsf{Piv}_x(A)\}_\alpha, \, \sigma_t \in \{\mathsf{Piv}_y(A)\}_\beta, \, \mathcal{G}_x=G \big)}_{=:p_{x,y,G}},
\end{align*}
where (as in Section \ref{ssec:CS}) for a static event $B$ we use the notation $B_\alpha=\{\eta\in \Omega_n : \alpha|\eta \in B\}$ (or similarly $\{B\}_\alpha$).\footnote{Remark: We consider $\{\mathsf{Piv}_x(A)\}_\alpha$ instead of just $\mathsf{Piv}_x(A)$ in order to apply Lemma \ref{lem:ins_tol} just below.}

By Lemma~\ref{lem:ins_tol},
\[
p_{x,y,G} \le \bigg(\frac{16}{\cit}\bigg)^{|G|+1}\cdot\tau^{|G|}\cdot\max\big\{\Pro\big(\sigma, \sigma_t \in \mathsf{Piv}_x(A)\big),\Pro\big(\sigma, \sigma_t \in \mathsf{Piv}_y(A)\big)\big\}.
\]
Recall the definition of $M$ at the beginning of Section \ref{sec:mix_gen} (N.B: the number of admissible values of $G$ such that $|G|\le k$ is less than $M^k$). The above inequality implies that
\begin{align*}
& -\frac{d}{dt}\Pro(\sigma,\sigma_t \in A)\\
& \le \frac{1}{2} \sum_{x \in \Lambda_n} \sum_{k\ge 0} M^k(k+1)\bigg(\frac{16}{\cit}\bigg)^{k+1}\tau^{k} \max_{\text{$y$ at graph}\atop \text{dist.\ $\le k$ from $x$}} \Pro\big(\sigma \in \mathsf{Piv}_y(A), \, \sigma_t \in \mathsf{Piv}_y(A)\big)\\
& \le \frac{1}{2}\sum_{y \in\Lambda_n} \Pro\big(\sigma \in \mathsf{Piv}_y(A), \, \sigma_t \in \mathsf{Piv}_y(A)\big) \sum_{k \ge 0} (k+1)\bigg(\frac{16}{\cit}\bigg)^{k+1}M^k\tau^k (2k+1)^2\\
& \lesssim \sum_{y \in \Lambda_n} \Pro\big(\sigma \in \mathsf{Piv}_y(A), \, \sigma_t \in \mathsf{Piv}_y(A)\big) \quad \text{by definition of $\tau$, see Equation~\eqref{eq:18}}.
\end{align*}
This concludes the general case. Let now concentrate on the case where $A$ is increasing and prove the lower bound. In this case (by still writing $f=1_A$), $f(\sigma^x)-f(\sigma)$ and $f(\sigma^{(x)}_t)-f(\sigma_t)$ always have the same sign, so
\begin{align*}
-\frac{d}{dt}\Pro(\sigma,\sigma_t \in A) & \ge \frac{1}{2}\sum_{x \in \Lambda_n} \E \Big[ c_x(\sigma) \big( f(\sigma^x)-f(\sigma) \big) \big( f(\sigma^{(x)}_t)-f(\sigma_t) \big) 1_{\mathcal{G}_x=\emptyset}\Big]\\
& = \frac{1}{2}\sum_{x \in \Lambda_n} \E \Big[ c_x(\sigma) 1_{\mathsf{Piv}_x(A)}(\sigma) 1_{\mathsf{Piv}_x(A)}(\sigma_t) 1_{\mathcal{G}_x=\emptyset}\Big]\\
& \gtrsim \sum_{x \in \Lambda_n} \Pro \big( \sigma,\sigma_t \in \mathsf{Piv}_x(A), \mathcal{G}_x=\emptyset\big) \quad \text{since $c_x\gtrsim 1$}\\
& \ge \frac{1}{2} \sum_{x \in \Lambda_n} \Pro \big( \sigma,\sigma_t \in \mathsf{Piv}_x(A)\big) \quad \text{by Lemma~\ref{lem:mix_dec} with $S=\{x\}$}. \qedhere
\end{align*}
\end{proof}
\section{\texorpdfstring{The events $\{\sigma,\sigma_t \in A\}$ for $\beta \le \beta_c$}{Dynamical events at  critical and supercritical temperature}}\label{sec:mix_sub}

Fix $\beta \le \beta_c$. Let $\cit$ and $\tau$ be the constants as defined in Section~\ref{ssec:defin-time-insert}. Recall that in this paper a static event is just a subset of $\Omega_n$.

\subsection{Spatial mixing}\label{ssec:mix}

In all this section, we fix some $n \ge 1$ and we let $(\sigma_s)_{s\ge 0}$ be the Glauber dynamics in $\Lambda_n$ (from Definition~\ref{defi:glauber}) that starts from $\sigma \sim \mu_n=\mu_{n,\beta}$.
In this section, we prove a spatial mixing result for events of the form $\{\sigma,\sigma_t\in A\}$.
As in Section \ref{ssec:1.2}, we write $\Lambda^{(\delta)}$ for the $(1+\delta)$-blown-up version of a rhombus $\Lambda$.
\begin{proposition}\label{prop:mix}
For every $\delta>0$, there exist two constants $c_\delta,C_\delta>0$ that do not depend on $n$ such that the following holds. Let $\Lambda$ be a rhombus such that $\Lambda^{(\delta)}\subseteq \Lambda_n$ and let $A$ and $B$ be two static events. If $A$ depends only on $\Lambda$ and $B$ depends only on $\Lambda_n\setminus\Lambda^{(\delta)}$, then for every $t \in[0,\tau]$,
\[
c_\delta\Pro(\sigma,\sigma_t\in A \cap B) \le \Pro(\sigma,\sigma_t\in A)\Pro(\sigma,\sigma_t \in B) \le C_\delta\Pro(\sigma,\sigma_t\in A \cap B).
\]
\end{proposition}
Our proof of Proposition \ref{prop:mix} is based on Lemma \ref{lem:mix_dec} and on another lemma that we state and prove now. Recall that the Glauber dynamics is defined by using the Poisson clocks $Y_x$ and the uniforms $U_{x,i}$, see Definition~\ref{defi:glauber}.
\begin{lemma}\label{lem:mix_Y}
For every $\delta>0$, there exist $c_\delta,C_\delta>0$ that do not depend on $n$ such that the following holds. Let $\Lambda$ be a rhombus such that $\Lambda^{(\delta)} \subseteq \Lambda_n$, let $\mathcal{A}$ be a dynamical event measurable with respect to $(\sigma,Y,U)_{|\Lambda}$ and let $\mathcal{B}$ be a dynamical event measurable with respect to $(\sigma,Y,U)_{|\Lambda_n\setminus\Lambda^{(\delta)}}$. Then,
\[
c_\delta\Pro(\mathcal{A}\cap\mathcal{B}) \le \Pro(\mathcal{A})\Pro(\mathcal{B}) \le C_\delta\Pro(\mathcal{A}\cap\mathcal{B}).
\]
\end{lemma}
\begin{proof}
In this proof, we use the symbol $\asymp_\delta$ to denote the same thing as $\asymp$ but with constants that may depend on $\delta$. To prove the claim, we are going to use the analogous result for the static Ising model (stated in Section \ref{ssec:inv}). By writing $1_\mathcal{A}$ as a function of $(\sigma,Y,U)_{|\Lambda}$ and $1_\mathcal{B}$ as a function of $(\sigma,Y,U)_{|\Lambda_n\setminus\Lambda^{(\delta)}}$, we have
\begin{align*}
\Pro(\mathcal{A}\cap\mathcal{B})& = \E \big[ \Pro\big(\mathcal{A}\cap\mathcal{B}\;\big|\; Y,U \big) \big]\\
& \asymp_\delta \E \Big[ \Pro \big( \mathcal{A} \:\big|\; Y,U \big) \Pro \big( \mathcal{B} \:\big|\; Y,U \big) \Big] \quad \text{by \eqref{eq:stat_mix_used} and since $ \sigma$ indep.\ of $(Y,U)$}\\
& = \E \Big[ \Pro \big( \mathcal{A} \:\big|\; (Y,U)_{|\Lambda} \big) \Pro \big( \mathcal{B} \:\big|\; (Y,U)_{|\Lambda_n\setminus\Lambda^{(\delta)}} \big) \Big]\\
& = \Pro(\mathcal{A}) \Pro(\mathcal{B}),
\end{align*}
where is the two last equalities we have used that $\sigma$, $(Y,U)_{|\Lambda}$ and $(Y,U)_{|\Lambda_n\setminus\Lambda^{(\delta)}}$ are independent.
\end{proof}
We can now prove Proposition \ref{prop:mix}. Below, given a rhombus $\Lambda$, we let $\mathsf{dec}=\mathsf{dec} (\partial\Lambda)$ and $\mathsf{dec}^{(\delta)}=\mathsf{dec}(\partial(\Lambda^{(\delta)}))$. (Recall the notation $\mathsf{dec}(S)$ from Definition~\ref{defi:dec}.)
\begin{proof}[Proof of Proposition \ref{prop:mix}]
Let $\delta>0$, $t \in [0,\tau]$ and let $\Lambda$ be a rhombus such that $\Lambda^{(\delta)}\subseteq\Lambda_n$. Let $A,B$ as in the statement. The proof is based on the observation that, if the sides of $\Lambda$ are larger than some $M_\delta>0$ (that we fix), then $\{\sigma,\sigma_t \in A\}\cap\mathrm{dec}$ is measurable with respect to $(\sigma,Y,U)$ restricted to $\Lambda^{(\delta/4)}$. This comes from the fact that, on $\mathsf{dec}$, there is a circuit made of non-green sites that is included in $\Lambda^{(\delta/4)}$ and surrounds $\Lambda$. Similarly, $\{\sigma,\sigma_t \in B\} \cap \mathsf{dec}^{(\delta)}$ is measurable with respect to $(\sigma,Y,U)$ restricted to $\Lambda_n \setminus \Lambda^{(\delta/2)}$. 
 Let us assume -- as a first step -- that the sides of $\Lambda$ are indeed larger than $M_\delta$. Then, Lemma \ref{lem:mix_Y} gives
\begin{align*}
&\Pro\big( \{\sigma,\sigma_t \in A\}\cap \mathsf{dec} \big) \cdot \Pro \big( \{\sigma,\sigma_t \in B\}\cap \mathsf{dec}^{(\delta)}\big)\\
&\asymp_\delta \Pro\big(\{\sigma,\sigma_t \in A\}\cap \mathsf{dec} \cap \{\sigma,\sigma_t \in B\}\cap \mathsf{dec}^{(\delta)} \big),
\end{align*}
where the symbol $\asymp_\delta$ denotes the same thing as $\asymp$ but with constants that may depend on $\delta$. We conclude by applying Lemma \ref{lem:mix_dec} three times (to $S=\partial\Lambda$, $S=\partial(\Lambda^{(\delta)})$ and S=$\partial\Lambda\cup\partial\Lambda^{(\delta)}$; N.B: we use that $\mathsf{dec}\cap\mathsf{dec}^{(\delta)}\subseteq\mathsf{dec}(\partial\Lambda\cup\partial\Lambda^{(\delta)})$).

\medskip

It remains to prove the result when the sides of $\Lambda$ are less than $M_\delta$, which we now assume. We also assume that $A$ is non-empty (otherwise the result is trivial). We obtain the left-hand-side of the inequality by writing
\[
\Pro(\sigma,\sigma_t \in A)\Pro(\sigma,\sigma_t\cap B)\ge c_0\Pro(\sigma,\sigma_t\in B)\ge c_0\Pro(\sigma,\sigma_t\in A\cap B),
\]
where $c_0=\min_{E \subseteq \Omega_n \text{ that dep.\ only on $\Lambda$}} \mu_n(E)^2$ (Remark: we have used that, as mentioned at the end of Section \ref{ssec:generator}, $\Pro(\sigma,\sigma_t\in A)\ge\mu_n(A)^2$). 

We prove the right-hand-side of the inequality by using the finite energy property for the pair $(\sigma,\sigma_t)$. More precisely, we first fix some $\alpha\in\{-1,+1\}^\Lambda$ such that $\eta_{|\Lambda}=\alpha \Rightarrow \eta \in A$. If $\psi\in \{-1,+1\}^{\Lambda_n\setminus\Lambda}$, we let $\alpha|\psi$ be the spin configuration that coincides with $\alpha$ in $\Lambda$ and with $\psi$ in $\Lambda_n\setminus \Lambda$. Moreover, we let $B'=\{\eta_{|(\Lambda_n\setminus\Lambda)} : \eta \in B\} \subseteq\{-1,+1\}^{\Lambda_n\setminus\Lambda}$. We have
\begin{align*}
\Pro(\sigma,\sigma_t\in A\cap B) & \ge \sum_{\psi,\psi'\in B'} \Pro(\sigma=\alpha|\psi,\sigma_t=\alpha|\psi')\\
& = \frac{1}{4^{|\Lambda|}}\sum_{\beta,\beta' \in \{-1,+1\}^\Lambda}\sum_{\psi,\psi'\in B'} \Pro(\sigma=\alpha|\psi,\sigma_t=\alpha|\psi')\\
& \ge \frac{\cit^{|\Lambda|}}{4^{|\Lambda|}}\sum_{\beta,\beta' \in \{-1,+1\}^\Lambda}\sum_{\psi,\psi'\in B'} \Pro(\sigma=\beta|\psi,\sigma_t=\beta'|\psi') \quad \text{by Cor.\ \ref{cor:ins}}\\
& = \frac{\cit^{|\Lambda|}}{4^{|\Lambda|}} \Pro(\sigma,\sigma_t \in B) \ge \frac{\cit^{|\Lambda|}}{4^{|\Lambda|}} \Pro(\sigma,\sigma_t \in A) \Pro(\sigma,\sigma_t \in B) .
\end{align*}
This concludes the proof.
\end{proof}
We will also need to decouple an event measurable with respect to $(\sigma,\sigma_t)$ and an event measurable with respect to $\sigma$. More precisely, we will need the following proposition.
\begin{proposition}\label{prop:mixbis}
For every $\delta>0$, there exist two constants $c_\delta,C_\delta>0$ that do not depend on $n$ such that the following holds. Let $\Lambda$ be a rhombus such that $\Lambda^{(\delta)}\subseteq \Lambda_n$ and let $A$ and $B$ be two static events. If $A$ depends only on $\Lambda$ and $B$ depends only on $\Lambda_n\setminus\Lambda^{(\delta)}$, then for any $t\in[0,\tau]$,
\[
c_\delta\Pro(\sigma \in A ,\sigma,\sigma_t\in B) \le \Pro(\sigma \in A)\Pro(\sigma,\sigma_t \in B) \le C_\delta\Pro(\sigma \in A ,\sigma,\sigma_t\in B)
\]
and
\[
c_\delta\Pro(\sigma,\sigma_t \in A ,\sigma \in B) \le \Pro(\sigma,\sigma_t \in A)\Pro(\sigma \in B) \le C_\delta\Pro(\sigma,\sigma_t \in A ,\sigma \in B).
\]
\end{proposition}
\begin{proof}
Let us prove the second estimate. The proof of the first one is the same. As in the proof of Proposition \ref{prop:mix}, we notice that it the sides of $\Lambda$ are large enough, then Lemma \ref{lem:mix_Y} implies that
\begin{align*}
&\Pro\big( \{\sigma,\sigma_t \in A\}\cap \mathsf{dec} \big) \Pro (\sigma\in B)\\
&\asymp_\delta \Pro\big(\{\sigma,\sigma_t \in A\}\cap \mathsf{dec} \cap \{\sigma \in B\} \big).
\end{align*}
Again as in the proof of Proposition \ref{prop:mix}, we conclude by applying Lemma~\ref{lem:mix_dec} and if the sides of $\Lambda$ are less than some constant then we conclude by applying Corollary~\ref{cor:ins} to every $x \in \Lambda$.
\end{proof}

\subsection{Combining spatial mixing and FKG inequality}

As in the previous section, we fix some $n \ge 1$ and we let $(\sigma_s)_{s\ge 0}$ be the Glauber dynamics in $\Lambda_n$ (from Definition~\ref{defi:glauber}) that starts from $\sigma \sim \mu_n=\mu_{n,\beta}$. (Recall that for all of Section \ref{sec:mix_sub} we have fixed some $\beta \le \beta_c$.)
By the dynamical FKG inequality (see Section \ref{ssec:dyn_fkg}), two increasing events measurable with respect to the pair $(\sigma,\sigma_t)$ are positively correlated. However, we will sometimes need to use that two events correlate well even if they are not both increasing everywhere. In practice, we will use the following proposition, whose proof is close to that of the spatial mixing property Proposition \ref{prop:mix_fkg} (and also relies on the dynamical FKG inequality).
We say that a static event $A\subseteq\Omega_n$ is increasing in some set $V\subseteq\Lambda_n$ if $1_A(\psi)\ge1_A(\eta)$ for every spin configurations $\eta,\psi\in \Omega_n$ such that $\psi$ coincides with $\eta$ in $\Lambda_n\setminus V$ and $\psi(x)\ge \eta(x)$ for every $x\in V$. 
\begin{proposition}\label{prop:mix_fkg}
For every $\delta>0$, there exist two constants $c_\delta,C_\delta>0$ that do not depend on $n$ such that the following holds. Let $\Lambda$ be a rhombus such that $\Lambda^{(\delta)}\subseteq \Lambda_n$ and let $A$ and $B$ be two static events. If $A$ is increasing and depends only on $\Lambda$ and $B$ is increasing in $\Lambda^{(\delta)}$, then for every $t \in[0,\tau]$,
\[
\Pro(\sigma,\sigma_t\in A \cap B) \ge c_\delta \Pro(\sigma,\sigma_t\in A)\Pro(\sigma,\sigma_t \in B).
\]
\end{proposition}
\begin{proof}
Let $\delta>0$ and let $\Lambda$ be a rhombus such that $\Lambda^{(\delta)}\subseteq\Lambda_n$. We first observe that if the sides of $\Lambda$ are less than some constant then  the proposition is a consequence of Corollary \ref{cor:ins} (applied to every $x \in \Lambda$ -- like in the proof of Proposition \ref{prop:mix}), so we assume that these sides are larger than some $M_\delta>0$ that we will fix during the proof. We define $\mathsf{dec}^{(\delta)}$ and $\mathsf{dec}$ as in the previous section, i.e.\ $\mathsf{dec}=\mathsf{dec}( \partial\Lambda)$ and $\mathsf{dec}^{(\delta)}=\mathsf{dec}(\partial(\Lambda^{(\delta)}))$. Moreover, we let $\Gamma$ be the outermost circuit in $\Lambda^{(\delta)}\setminus\Lambda^{(\delta/2)}$ that surrounds $\Lambda^{(\delta/2)}$ and is made of non-green sites (recall that a green site is a site whose clock has rung before time $\tau$). This circuit exists on the event $\mathsf{dec}^{(\delta)}$ as soon as $M_\delta$ is large enough, which we assume. If this circuit does not exist, we let $\Gamma=\emptyset$ (this is a convention that has no importance). If $\gamma$ is a circuit that surrounds $\Lambda^{(\delta/2)}$, we let $\gamma_{\mathrm{out}}$ denote the union of $\gamma$ and of all the points of $\Lambda_n$ that are separated from $\Lambda^{(\delta/2)}$ by $\gamma$ (and we use the convention $\gamma_\mathrm{out}=\Lambda_n$ if $\gamma=\emptyset$). We have
\begin{align}\label{eq:mixfkg0}
\Pro(\sigma,\sigma_t \in A\cap B) & = \E \Big[  \Pro\big( \sigma,\sigma_t \in A\cap B \; \big| \; \Gamma, (\sigma,\sigma_t)_{|\Gamma_\mathrm{out}} \big)\Big]\nonumber\\
& \ge \E \Big[  \Pro\big( \sigma,\sigma_t \in A\cap B \; \big| \; \Gamma, (\sigma,\sigma_t)_{|\Gamma_\mathrm{out}} \big)1_{\{\Gamma \neq \emptyset\}} \Big].
\end{align}
Let $\gamma$ and $(\eta,\psi)$ be possible values for $\Gamma$ and $(\sigma,\sigma_t)_{|\Gamma_\mathrm{out}}$ respectively, with $\gamma \neq \emptyset$, and let
\[
\mathcal E_{\gamma,\eta,\psi}=\{\Gamma=\gamma \} \cap \big\{ (\sigma,\sigma_t)_{|\Gamma_\mathrm{out}}=(\eta,\psi)\big\}.
\]
We claim that we have the following due to the spatial Markov property of the static Ising model: Let $\gamma_\mathrm{int}=\Lambda_n\setminus\gamma_\mathrm{out}$ be the region surrounded by $\gamma$. Then, conditionally on $\mathcal{E}_{\gamma,\eta,\psi}$, the law of $(\sigma_s)_{s\in[0,\tau]}$ in $\gamma_\mathrm{int}$ is that of the Glauber dynamics in this set with boundary conditions given by $\eta_{|\gamma}$, and starting from its invariant probability measure.

To prove this claim, we first observe that it is true if we rather condition on the event $\mathcal{E}'_{\gamma,\eta,\psi}$ that $\sigma$ coincides with $\eta$ in $\gamma$ and that no clock in $\gamma$ has rung before time $\tau$. The event $\mathcal{E}_{\gamma,\eta,\psi}$ equals the event $\mathcal{E}_{\gamma,\eta,\psi}'$ intersected with an event of the $\sigma$-algebra generated $(\sigma,Y,U)_{|\gamma_\mathrm{out}}$. Moreover, conditionally on $\mathcal{E}'_{\gamma,\eta,\psi}$, the process $(\sigma_s)_{s\in[0,\tau]}$ in $\gamma_\mathrm{int}$ is independent of $(\sigma,Y,U)_{|\gamma_\mathrm{out}}$. As a result, the law of $(\sigma_s)_{s\in[0,\tau]}$ in $\gamma_\mathrm{int}$ conditionally on $\mathcal{E}_{\gamma,\eta,\psi}$ is the same as conditionally on~$\mathcal{E}'_{\gamma,\eta,\psi}$. This proves the claim.

\medskip

We are now going to apply the dynamical FKG inequality stated in Section~\ref{ssec:dyn_fkg}: by the claim just above and the dynamical FKG inequality for the Glauber dynamics in $\gamma_\mathrm{int}$ with boundary conditions given by $\eta_{|\gamma}$ (see Remark \ref{rem:dyn_fkg}), we have\footnote{Here we also use that there exist $A',B',A'',B''\subseteq \{-1,+1\}^{\gamma_\mathrm{int}}$ that are increasing, whose definition depends on $\gamma$, $\eta$ and $\psi$, and that satisfy that, on the event $\mathcal E_{\gamma,\eta,\psi}$, $\{\sigma,\sigma_t \in A\} = \{ (\sigma,\sigma_t)_{|\gamma_\mathrm{int}}\in A'\times A''\}$ and  $\{\sigma,\sigma_t \in B\} = \{ (\sigma,\sigma_t)_{|\gamma_\mathrm{int}}\in B'\times B''\}$.}
\begin{equation}\label{eq:mixfkg1}
\Pro\big( \sigma,\sigma_t \in A\cap B \; \big| \; \mathcal E_{\gamma,\eta,\psi} \big) \ge \Pro\big( \sigma,\sigma_t \in A \; \big| \; \mathcal E_{\gamma,\eta,\psi} \big) \Pro\big( \sigma,\sigma_t \in B \; \big| \; \mathcal E_{\gamma,\eta,\psi} \big).
\end{equation}


Now, we are going to erase the conditioning in the first term of the right-hand-side of \eqref{eq:mixfkg1}. 
 We observe that, if $M_\delta$ large enough -- which we assume --, then the event $\{ \sigma,\sigma_t \in A\} \cap \mathsf{dec}$ is measurable with respect to $(\sigma,Y,U)$ restricted to $\Lambda^{(\delta/4)}$. Moreover, $\mathcal E_{\gamma,\eta,\psi}$ is measurable with respect to $(\sigma,Y,U)$ restricted to the complement of $\Lambda^{(\delta/2)}$. As a result (by using Lemma~\ref{lem:mix_Y} in the second inequality), there exists a constant $c_\delta>0$ such that
\begin{align}\label{eq:mixfkg2}
\Pro\big( \sigma,\sigma_t \in A \; \big| \; \mathcal E_{\gamma,\eta,\psi} \big) & \ge \frac{\Pro\big(\{ \sigma,\sigma_t \in A\} \cap \mathsf{dec} \cap \mathcal{E}_{\gamma,\eta,\psi} \big)}{\Pro\big( \mathcal E_{\gamma,\eta,\psi}\big)}\nonumber\\
& \ge c_\delta\Pro\big( \{\sigma,\sigma_t \in A\}\cap \mathsf{dec}\big).
\end{align}
By combining \eqref{eq:mixfkg0}, \eqref{eq:mixfkg1} and \eqref{eq:mixfkg2}, we obtain that
\begin{align*}
\Pro\big(\sigma,\sigma_t \in A\cap B\big)  & \ge c_\delta\Pro\big( \{\sigma,\sigma_t \in A\}\cap \mathsf{dec}\big) \cdot \Pro\big(\{ \sigma,\sigma_t \in B \}\cap \{\Gamma \neq \emptyset\} \big)\\
&  \ge c_\delta\Pro\big( \{\sigma,\sigma_t \in A\}\cap \mathsf{dec}\big) \cdot \Pro\big(\{ \sigma,\sigma_t \in B \}\cap \mathsf{dec}^{(\delta)}  \big).
\end{align*}
We conclude by applying Lemma \ref{lem:mix_dec}.
\end{proof}
\subsection{Quasi-invariance by translation}\label{ssec:inv_dyn}
As explained at the end of Section \ref{ssec:1.2}, since we work in finite volume, we do not have translation invariance but we rather have a quasi-invariance by translation property that we state and prove now.
Note that in the following statement, contrary to the previous sections, $(\sigma_s)_{s \ge 0}$ is the Glauber dynamics in $\Lambda_{2n}$ starting from $\sigma\sim\mu_{2n}$ (so we work in $\Lambda_{2n}$ rather than $\Lambda_n$). Recall that if $m\ge 1$, $\eta \in \{-1,1\}^{\Lambda_{2m}(x)}$ and $x\in \T$, we let $T_{-x}\eta = \eta(x+\cdot) \in \Omega_{2m}$.
\begin{proposition}\label{prop:quasi-inv}
There exist $c,C>0$ such that the following holds. Let $n\ge m \ge 1$, let $(\sigma_s)_{s \ge 0}$ be the Glauber dynamics in $\Lambda_{2n}$ starting from $\sigma\sim\mu_{2n}$ and let $(\sigma_s')_{s\ge 0}$ denote the Glauber dynamics in $\Lambda_{2m}$ starting from $\sigma'\sim\mu_{2m}$. Moreover, let $A\subseteq\Omega_{2m}$ be a static event that depends only on $\Lambda_m$ and let $x \in \T$ such that $\Lambda_m(x)\subseteq \Lambda_n$. Then, for any $t \in[0,\tau]$,
\begin{equation}\label{eq:inv_glauber}
c\Pro(\sigma',\sigma_t' \in A) \le \Pro\big(\sigma,\sigma_t \in T_xA\big) \le C\Pro(\sigma',\sigma'_t \in A),
\end{equation}
where $T_xA$ is the event $A$ translated by $x$ i.e.\ $T_xA=\{ \eta \in \Omega_{2n} : T_{-x}(\eta_{|\Lambda_{2m}(x)}) \in A \}$.
\end{proposition}
\begin{proof}
Let $M>0$ be a constant that we will fix during the proof. We first notice that it is sufficient to prove the proposition under the assumption that $m$ is larger than $M$.
So let $m\ge M$. If $f:\Omega_{2m}\rightarrow \R$, we let $T_xf : \Omega_{2n} \rightarrow \R$ be defined by $T_xf(\eta)=f(T_{-x}(\eta_{|\Lambda_{2m}(x)}))$. By spatial mixing of the static Ising model, for every function $f:\Omega_{2m}\to \mathbb R_+$ that only depends on $\Lambda_{1.5m}$ we have 
\begin{equation}
\label{eq:8}
\mathbb E[f(\sigma')] \asymp \mathbb E\big[T_xf(\sigma)\big].
  \end{equation}
(More precisely, one can obtain the above by applying \eqref{eq:stat_mix_box} to the event $\{$the spin configuration restricted to $\Lambda_{1.5m}$ equals $\eta \} \subseteq \Omega_{2m}$ for every $\eta \in\Omega_{1.5m}$.)
\smallskip
We choose $f : \eta \in \Omega_{2m} \mapsto \mathbb P(\{\sigma',\sigma_t'\in A\}\cap \mathsf{dec}'(\Lambda_{m})|\sigma'=\eta)$, where the notation $\mathsf{dec}'$ is used to denote the decoupling events from Definition \ref{defi:dec} for the Glauber dynamics $(\sigma'_s)_{s\ge 0}$. The function $f$ indeed only depends on $\Lambda_{1.5m}$ as soon as $M$ is large enough (which we assume). Applying the displayed equation above gives
  \begin{equation}
        \mathbb P\big(\{\sigma',\sigma'_t\in A\}\cap \mathsf{dec'}(\Lambda_{m})\big)\asymp   \mathbb P\big(\{\sigma,\sigma_t\in T_xA\}\cap \mathsf{dec}(\Lambda_{m}(x))\big).
      \end{equation}
The proof then follows from Lemma~\ref{lem:mix_dec}.
\end{proof}
\section{Percolation arguments for $\beta\le \beta_c$}
\label{sec:quasimulti}
Fix $\beta \le \beta_c$ for the whole section. Let $\tau$ be the time constant as defined in Section~\ref{ssec:defin-time-insert}. 

\medskip

As in \cite{TV23}, the following `dynamical four-arm probability' is at the core of our analysis:
\begin{equation}\label{eq:pi}
\pi_n(t)=\Pro\big(\sigma,\sigma_t \in A_4(n)\big),
\end{equation}
where $(\sigma_s)_{s \ge 0}$ is the Glauber dynamics in $\Lambda_{2n}$ starting from $\sigma\sim \mu_{2n}$. (Recall that $A_4(n)$ is the four-arm event defined in Section \ref{ssec:sharp}.)

\medskip

In this section, we prove the quasi-multiplicativity property for $\pi_n(t)$ and the multiscale analogous quantities $\pi_{m,n}(t)$ that we define just below. Then, we use this property and the differential formula Proposition \ref{prop:diff} to prove differential inequalities for $\pi_n(t)$ and $\Pro(\sigma,\sigma_t \in \mathrm{Cross}_n)$.
\subsection{Quasi-multiplicativity}
The quasi-multiplicativity property  is a key tool in the study of arm events. It was first established by Kesten \cite{Kes87} in the context of Bernoulli percolation, to prove  scaling relations between critical and near-critical exponents. See also  \cite{Wer07,Nol08,SS10, DMT21} for alternative proofs and additional applications in this context.  For dynamical arm events (i.e.\ events of the kind $\{$a $4$-arm event holds at times $0$ and $t\}$), quasi-multiplicativity was established by Garban, Pete and Schramm \cite{GPS10} for Bernoulli percolation under i.i.d.\ dynamics (see also the appendix of \cite{RT24}). In this section, we prove the analogue of the quasi-multiplicativity of  \cite{GPS10}   in the case of the (high temperature) Ising model,  following the approach  of \cite{DMT21,RT24} and combining it with the Cauchy--Schwarz trick (Lemma \ref{lem:CS}).
\begin{definition}
Let $n \ge m \ge 2$. We let $A_4(m,n) \subset \Omega_{2n}$ denote the event that there are four paths of alternating signs from $\partial \Lambda_m$ to $\partial\Lambda_n$ included in the annulus $\Lambda_n\setminus\Lambda_{m-1}$ (Remark: We can -- and we will -- see $A_4(m,n)$ as a subset of $\Omega_N$ for any $N\ge n$). We let
\[
\alpha_{m,n}=\mu_{2n}(A_4(m,n)) \quad \text{and} \quad \pi_{m,n}(t)=\Pro\big(\sigma,\sigma_t \in A_4(m,n)\big),
\]
where $(\sigma_s)_{s \ge 0}$ is the Glauber dynamics in $\Lambda_{2n}$ starting from $\sigma\sim \mu_{2n}$.
For $n\ge 1$, we set
\[
\alpha_{1,n}=\alpha_n \quad \text{and} \quad \pi_{1,n}(t)=\pi_n(t),
\]\end{definition}
\begin{proposition}[Quasi-multliplicativity property]\label{prop:quasi_multi}
There exist $c,C>0$ such that, for any $n\ge m \ge k \ge 1$ and $t \in[0,\tau]$,
\[
c\pi_{k,n}(t) \le \pi_{k,m}(t)\pi_{m,n}(t) \le C\pi_{k,n}(t).
\]
\end{proposition}
\begin{remark}
By Proposition \ref{prop:quasi_multi} and the box-crossing property, there exists $c>0$ such that for every $n\ge m \ge 2$ and $t\in[0,\tau]$, $\pi_{m/2,2n}(t) \ge c\pi_{m,n}(t)$.
\end{remark}
\subsubsection{Separation of arms}
The main ingredient in the proof of the quasi-multiplicativity property is the proof of the so-called separation of arms property.
\begin{definition}
Let $n \ge 100$ and let $x_1,\dots,x_4$ denote the centers of the left, bottom, right and top sides of $\Lambda_n$. We let $A_4^{\textup{sep}}(m,n) \subset A_4(m,n)$ denote the (static) event that for every $i \in \{1,\dots,4\}$ odd (resp.\ even) there exists a path from $\partial \Lambda_m$ to $\partial \Lambda_n$ included in the annulus $\Lambda_n \setminus \Lambda_{m-1}$ that (a) does not intersect the three rhombii $\Lambda_{n/10}(x_j)$, $j \neq i$, (b) is only made of $+1$ (resp.\ $-1$) spins and (c) ends on $\Lambda_{n/20}(x_i)\cap\partial \Lambda_n$. Moreover, we let
\[
\pi_{m,n}^\mathrm{sep}(t)=\Pro \big( \sigma,\sigma_t \in A_4^{\textup{sep}}(m,n) \big),
\]
where $(\sigma_s)_{s\ge 0}$ is the Glauber dynamics in $\Lambda_{2n}$ starting from $\sigma\sim\mu_{2n}$.
\end{definition}
\begin{proposition}[Separation of arms]\label{prop:sep}
There exists $c>0$ such that, for every $m\ge 1$,  every $n \ge (4m)\vee 100$ and every $t \in[0,\tau]$,
\[
\pi_{m,4n}^\mathrm{sep}(t) \ge c \pi_{m,n}(t).
\]
\end{proposition}
Roughly speaking, we use standard arguments  to deduce the proposition above  from  the box-crossing property and the spatial mixing and FKG results Propositions~\ref{prop:mix} and~\ref{prop:mix_fkg}.  The main new difficulty that was not present in previous approaches and that we had to overcome is of combinatorial nature. We are considering 4 arms at time 0 and at time $t$. When we localize the arms at time 0 and at time $t$, there is no reason a priori that the 4 arms land at the same places at both times, and some combinatorial obstructions may appear. The Cauchy--Schwarz trick (Lemma~\ref{lem:CS}) solves this issue and asserts that the most likely case is when the  arms land roughly at the same place at time $0$ and at time $t$. This is detailed in the proof of Lemma~\ref{lem:ext_sep}.
We decompose the proof into three lemmas (Lemmas \ref{lem:ext}, \ref{lem:3arm} and \ref{lem:ext_sep}). 
The first lemma is concerned with one of the main advantage of  separated arms: since the arms are at macroscopic distance of each other, we can use standard box-crossing constructions to extend them to a larger scale.
\begin{lemma}\label{lem:ext}
There exists $c>0$ such that, for every $n,m\ge 1$ that satisfy $n \ge (4m)\vee 100$ and any $t \in[0,\tau]$,
\[
\pi_{m,4n}^\mathrm{sep}(t) \ge c \pi_{m,n}^\mathrm{sep}(t).
\]
\end{lemma}
\begin{proof}
We let $x_1,\dots,x_4$ (resp.\ $y_1,\dots,y_4$) denote the centers of the left, bottom, right and top sides of $\Lambda_n$ (resp.\ $\Lambda_{4n}$). Let $R_1= \{k+\ell e^{i\pi/3},-4n\le k\le -n-n/20, -n/20\le \ell \le n/20\}$ be a long horizontal rhombus  that intersects the left boundary of $\Lambda_{4n}$ and is exactly adjacent to the rhombus $\Lambda_{n/20}(x_1)$. Similarly, for each $i\in \{2,3,4\}$, consider a rhombus $R_i$  that connects the rhombus $\Lambda_{n/20}(x_i)$ to the boundary of $\Lambda_{4n}$.  Since we will apply the spatial mixing property, we emphasize that those rhombii  are at macroscopic distance (roughly  $n/20$) from $\Lambda_n$. For every  $i\in \{1,\ldots,4\}$,  we also define two static events $A_i$ and $B_i$ as follows: (In these two events, the circuits/paths are made of $+1$ spins if $i$ is odd an $-1$ spins if $i$ is even)
\begin{itemize}[noitemsep]
\item We let $A_i$ denote the event that there is a circuit in $\Lambda_{n/15}(x_i) \setminus \Lambda_{n/20}(x_i)$ that surrounds the inner rhombus;
\item We  let $B_i$ be the event that there is a path included in $R_i$ from $\partial\Lambda_{n/20}(x_i)$ to $\Lambda_{(4n)/20}(y_i) \cap \partial\Lambda_{4n} $.
\end{itemize}
The events above are constructed in such  a way that the following inclusion of static events holds:
\[
A_4^{\textup{sep}}(m,n) \cap \bigg( \bigcap_{i=1}^4 A_i \cap B_i \bigg) \subseteq A_4^{\textup{sep}}(m,4n).
\]
Let $(\sigma_s)_{s\ge 0}$ denote the Glauber dynamics in $\Lambda_{8n}$ starting from $\sigma\sim\mu_{8n}$. By applying four times Proposition \ref{prop:mix_fkg} (with $\Lambda=\Lambda_{n/15}(x_i)$, $i=1,\dots,4$ and $\delta=1/3$),\footnote{Here, we use that, for every $k=1,\dots,4$, $A_4^{\textup{sep}}(m,n) \cap \bigcap_{i=1}^4 B_i \cap \bigcap_{i=k+1}^4 A_i$ and $A_k$ are both increasing (resp.\ decreasing) in $\Lambda_{n/10}(x_k)$ if $k$ is odd (resp.\ even).} we obtain that
\begin{align*}
\pi_{m,4n}^\mathrm{sep}(t)&=\Pro \big( \sigma,\sigma_t \in A_4^{\textup{sep}}(m,4n) \big)\\
& \ge \Pro \big( \sigma,\sigma_t \in A_4^{\textup{sep}}(m,n) \cap \bigcap_{i=1}^4 B_i \cap \bigcap_{i=1}^4 A_i \big)\\
& \gtrsim \Pro \big( \sigma,\sigma_t \in A_4^{\textup{sep}}(m,n) \cap \bigcap_{i=1}^4 B_i \cap \bigcap_{i=2}^4 A_i \big) \Pro \big( \sigma,\sigma_t \in A_1 \big)\\
& \gtrsim \dots\\
& \gtrsim \Pro \big( \sigma,\sigma_t \in A_4^{\textup{sep}}(m,n) \cap \bigcap_{i=1}^4 B_i \big) \prod_{i=1}^4 \Pro \big( \sigma,\sigma_t \in A_i \big).
\end{align*}
By the above and Proposition \ref{prop:mix} (applied to $\Lambda=\Lambda_n$), we have
\begin{align*}
\pi_{m,4n}^\mathrm{sep}(t)& \gtrsim \Pro \big( \sigma,\sigma_t \in A_4^{\textup{sep}}(m,n) \big) \Pro \big(\sigma,\sigma_t \in \bigcap_{i=1}^4 B_i \big) \prod_{i=1}^4 \Pro \big( \sigma,\sigma_t \in A_i \big)\\
&\gtrsim \Pro \big( \sigma,\sigma_t \in A_4^{\textup{sep}}(m,n) \big),
\end{align*}
where in the last inequality we have used the box-crossing property, together with the inequality $\Pro(\sigma,\sigma_t\in A)\ge \mu_{8n}(A)^2$ (see the remarks at the end of Section \ref{ssec:generator}). The proof is not totally over, because $(\sigma_s)_{s\ge 0}$ is the Glauber in dynamics in $\Lambda_{8n}$ and not in $\Lambda_{2n}$. But by Proposition \ref{prop:quasi-inv} (with $x=o$), the above is at least a constant times~$\pi_{m,n}^\mathrm{sep}(t)$. This ends the proof.
\end{proof}
The main idea behind the proof of the separation of arms is that two interfaces land typically at macroscopic distance of each other at the boundary of a box.  Technically, this is due to the fact that two  interfaces landing close to each others imply the existence of 3 boundary arms, which is unlikely to occur.   In order to control this quantitatively, we introduce below a typical event excluding 3-arms at the boundary of the box.    
\begin{definition}
Given $n\ge 100$ and $\delta\in(0,1/100)$ such that $n\ge 1/\delta$, we let $\textup{Sep}_n^\delta \subset \Omega_{2n}$ be the static event defined by the fact that there is no $x\in \partial\Lambda_n$ such that there exist three paths of alternating sign from $\partial \Lambda_{16\delta n}(x)$ to $\partial \Lambda_{n/2}$ that are included in $\Lambda_n$ (Remark: We can -- and we will -- see $\textup{Sep}_n^\delta$ as a subset of $\Omega_N$ for any $N\ge n$). 
\end{definition}
We will use two important properties  of  $\textup{Sep}_n^\delta$. First, Lemma~\ref{lem:3arm} asserts that it is a typical event: its probability is closed to $1$ provided the separation parameter $\delta$ is small enough. Second, its occurrence    forces arms to separate at scale $n$. This is formally expressed in Claim~\ref{cl:sep} below, and -- together with the Cauchy--Schwarz trick -- is the central argument in the proof of Lemma~\ref{lem:ext_sep}.
   
\begin{lemma}\label{lem:3arm}
There exists $C>0$ such that, for every $\delta\in(0,1/100)$ and $n \ge 1/\delta$,
\[
\mu_{2n}(\textup{Sep}_n^\delta)\ge 1-C \delta.
\]
\end{lemma}
\begin{proof}
Let $D\subset \partial\Lambda_n$  be a subset of size $\asymp \frac1\delta$  such that, each point  $x\in D$ is at (Euclidean) distance less than $16\delta n$ from the other points of $D$. We observe that if $\textup{Sep}_n^\delta$ does not hold then there exits $x \in D$ such that a $3$-arm event holds in a half-plane around $x$ from distance $32\delta n$ to distance $n/2$. More precisely, if $\textup{Sep}_n^\delta$ does not hold then there exists $x \in D$ such that there exist three arms of alternating color included in $\Lambda_n$ from $\partial \Lambda_{32\delta n}(x)$ to $\partial \Lambda_{n/2}$. By using that the exponent of the $3$-arm event is (at least) $2$ (see Section \ref{app:3arm} in the appendix), we obtain that -- for fixed $x$ -- the probability of this event is less than a constant times $\delta^2$. (Remark: here, we have also used quasi-invariance by translation for the static model, i.e.\ \eqref{eq:stat_mix_box}.) We obtain the desired result by applying the union bound.
\end{proof}
Let us state and prove the third lemma of this subsection.
\begin{lemma}\label{lem:ext_sep}
  For every $\delta\in(0,1/100)$ there exists $c_\delta>0$ such that for any $m,n\ge 1$ that satisfy $n\ge (4m)\vee (1/\delta)$ and any $t\in[0,\tau]$,
\[
\pi_{m,4n}^\mathrm{sep}(t) \ge c_\delta \Pro\big( \sigma,\sigma_t \in A_4(m,n)\cap \textup{Sep}_n^\delta\big),
\]
where $(\sigma_s)_{s \ge 0}$ is the Glauber dynamics in $\Lambda_{2n}$ starting from $\sigma\sim\mu_{2n}$. 
\end{lemma}
In order to prove this lemma, we define the $\delta$-well-separated events $A_4^{\delta,X}(m,n)$.
\begin{definition}
Let $\delta \in (0,1/100)$ and $n \ge 1/\delta$. We let $D_n^\delta \subseteq \partial \Lambda_n$ be a set of points roughly equally spaced, of cardinality $\asymp 1/\delta$. More precisely, we just ask that each $x\in D_n^\delta$ is at distance between $\delta n$ and $2\delta n$ from $D_n^\delta \setminus \{x\}$. Moreover,
\begin{itemize}
\item If $X=(x_1,\dots,x_4) \in (D_n^\delta)^4$, we say that $X$ is $(\delta,n)$-well-separated if $x_1,\dots,x_4$ are in counterclockwise order and $\Lambda_{4\delta n}(x_i)\cap \Lambda_{4\delta n}(x_j)=\emptyset$ for every $i \neq j$;
\item We define the (static) events $A_4^{\delta,X}(m,n)\subset A_4(m,n)$ as follows. Let $m\ge 1$ such that $n\ge (4m)\vee(1/\delta)$ and let $X=(x_1,\dots,x_4)\in(D_n^\delta)^4$ be a $(\delta,n)$-well-separated tuple. We say that $A_4^{\delta,X}(m,n)$ holds if for every $i \in \{1,\dots,4\}$ odd (resp.\ even) there exists a path from $\partial \Lambda_m$ to $\partial \Lambda_n$ that is included in $\Lambda_n \setminus (\cup_{j\neq i}\Lambda_{4\delta n}(x_j) \cup \Lambda_{m-1})$, that is only made of $+1$ (resp.\ $-1$) spins, and whose endpoint belongs to $\Lambda_{2\delta n}(x_i)\cap\partial \Lambda_n$.
\end{itemize}
\end{definition}
\begin{proof}[Proof of Lemma \ref{lem:ext_sep}]
We start with the following static and deterministic claim. 
\begin{claim}\label{cl:sep}
Let $\delta\in(0,1/100)$ and $m,n\ge 1$ such that $n\ge (4m)\vee (1/\delta)$. Then,
\[
A_4(m,n)\cap \textup{Sep}_n^\delta \subseteq A_4^\delta(m,n):=\cup_X A_4^{\delta,X}(m,n),
\]
where the union is on all $X\in(D_n^\delta)^4$ that are $(\delta,n)$-well-separated.
\end{claim}
Before proving the claim, let us prove the lemma. Let $(\sigma_s)_{s \ge 0}$ be the Glauber dynamics in $\Lambda_{2n}$ starting from $\sigma\sim\mu_{2n}$. Let $X\in(D_n^\delta)^4$ be the $(\delta,n)$-well-separated tuple that maximizes $W\mapsto\Pro(\sigma_,\sigma_t \in A_4^{\delta,W}(m,n))$. By the `Cauchy--Schwarz trick' Lemma \ref{lem:CS},
\[
\Pro \big( \sigma,\sigma_t \in A_4^{\delta,X}(m,n) \big) \ge \frac{1}{N_{\delta,n}} \sum_{W,Z} \Pro \big( \sigma \in A_4^{\delta,W}(m,n), \sigma_t \in A_4^{\delta,Z}(m,n)  \big),
\]
where the sum is on every $(\delta,n)$-well-separated tuples $W,Z\in(D_n^\delta)^4$ and $N_{\delta,n}\asymp1/\delta^8$ is the number of possible choices for $(W,Z)$. As a result (recall the definition of $A_4^\delta(m,n)$ in the statement of the claim),
\[
\Pro \big( \sigma,\sigma_t \in A_4^{\delta,X}(m,n) \big) \gtrsim\delta^8 \Pro\big(\sigma,\sigma_t \in A_4^\delta(m,n) \big) \overset{\textup{Claim \ref{cl:sep}}}{\ge} \delta^8 \Pro\big( \sigma,\sigma_t \in
A_4(m,n)\cap \textup{Sep}_n^\delta \big).
\]
Let $(\sigma_s')_{s\ge 0}$ be the Glauber dynamics in $\Lambda_{8n}$ starting from $\sigma'\sim\mu_{8n}$. By the above (and Proposition \ref{prop:quasi-inv} with $x=o$), it is sufficient to prove that
\[
\Pro\big(\sigma',\sigma_t'\in A_4^{\textup{sep}}(m,4n)\big) \ge c_\delta\Pro \big( \sigma',\sigma_t' \in A_4^{\delta,X}(m,n) \big).
\]
The proof of the latter is similar to the one of Lemma \ref{lem:ext} where we used four rhombi to extend four separated arms at scale $n$ to scale $4n$. Here we can repeat the argument. The only difference is that one needs to replace the four rhombi by four well-chosen disjoint `corridors' that connect the four boxes associated to $X$ at scale $n$ to the four centered boxes (as in the definition of $A_4^{\textup{sep}}$) at scale $n$. We omit the precise details of this construction, which is standard in separation arguments.

\medskip

Let us finish the proof of the lemma by showing the claim. To this purpose, we are going to consider the interfaces that cross $\Lambda_n \setminus \Lambda_m$. An interface is a continuous path in $\R^2$ that is included on the union of edges of the hexagonal lattice (which is dual to the triangular lattice) and that has only $+1$ sites on one of its side and only $-1$ sites on the other side. If $A_4(m,n)$ holds, then there are at least four interfaces that are disjoint (in the sense that the continuous paths are disjoint) and that cross this annulus. Assume that $A_4(m,n)$ holds, let $\gamma_1,\dots,\gamma_4$ be four consecutive interfaces and let $z_1,\dots,z_4\in\partial \Lambda_n$ be four vertices of the triangular lattice that are adjacent to the endpoints of $\gamma_1,\dots,\gamma_4$ respectively and whose signs are $+,-,+,-$ respectively.
We observe that if $\textup{Sep}^\delta_n$ holds, then for every $i\in\{1\dots,4\}$, $\gamma_i \cap \big(\cup_{j\neq i}\Lambda_{8\delta n}(z_j)\big)=\emptyset$ (by this, we mean that the set of sites adjacent to $\gamma_i$ does not intersect $\cup_{j\neq i}\Lambda_{8\delta n}(z_j)$). One concludes the proof of the claim by letting $x_i$ be the point of $D_n^\delta$ that is closest to $z_i$ and noting that $A_4^{\delta,X}(m,n)$ holds, where $X=(x_1,\dots,x_4)$. 
\end{proof}
We can now prove the arm separation property.
\begin{proof}[Proof of Proposition \ref{prop:sep}] Let $\delta\in(0,1/100)$ (that will be  chosen small enough at the  end of the proof) and let $m,n\ge 1$ such that $n \ge (4m) \vee (1/\delta)$. Let $(\sigma_s)_{s\ge 0}$ denote the Glauber dynamics in $\Lambda_{2n}$ starting from $\sigma\sim\mu_{2n}$. Either the separation event holds both at time $0$ and time $t$, or it must fail at time $0$ or at time $t$. By using the union bound and reversibility  of the Glauber dynamics, we get
\begin{align*}
\pi_{m,n}(t) & \le \Pro \big( \sigma,\sigma_t \in A_4(m,n) \cap \textup{Sep}_n^{\delta} \big) + 2\Pro \big( \sigma,\sigma_t \in A_4(m,n), \sigma\notin\textup{Sep}_n^\delta \big).\\
 & \le \Pro \big( \sigma,\sigma_t \in A_4(m,n) \cap \textup{Sep}_n^{\delta} \big) + 2\Pro \big( \sigma,\sigma_t \in A_4(m,n/4), \sigma\notin\textup{Sep}_n^\delta \big).
\end{align*}
By Lemma~\ref{lem:ext_sep}, the first term of the last line above is  at most $C_\delta\pi_{m,4n}^\mathrm{sep}(t)$, where $C_\delta$ is a constant depending on $\delta$ only. By the spatial mixing property Proposition \ref{prop:mixbis} and Lemma \ref{lem:3arm}, the second term is at most $ C'\delta\pi_{m,n/4}(t)$ for some constant $C'>0$. We obtain
  \begin{equation}
    \label{eq:5}
    \pi_{m,n}(t) \le C_\delta\pi_{m,4n}^\mathrm{sep}(t) + C'\delta \pi_{m,n/4}(t),
  \end{equation}
Writing $r_{m,n}=\pi_{m,n}(t)/\pi_{m,4n}^\mathrm{sep}(t)$,  the displayed equation  above and Lemma \ref{lem:ext} gives
\[
r_{m,n}\le C_\delta+C''\delta r_{m,n/4}
\]
for every $m\ge 1$ and $n\ge (4m) \vee (1/\delta)$ and for some constant $C''>0$. We choose $\delta$  small enough so that $C''\delta \le 1/2$. A direct induction implies that for every $n,m\ge 1$ such that $n \ge (4m)\vee(1/\delta)$, we have
\begin{align*}
r_{m,n}\le 2C_\delta + \max_{k\ge (4m)\vee(1/\delta)} r_{m,k} & \le 2C_\delta + \max_{k\ge (4m)\vee(1/\delta)} \frac{1}{\pi_{m,4k}^\mathrm{sep}(t)}\\
& \le 2C_\delta + \max_{k\ge (4m)\vee(1/\delta)} \frac{1}{\mu_{8k}\big(A_4^\mathrm{sep}(m,4k)\big)^2}.
\end{align*}
The above is less than some constant (thanks to the box-crossing property \eqref{eq:bxp}), which gives the desired result when $n\ge 4m\vee(1/\delta)$. The extension to $ n\ge 4m\vee 100$ follows by possibly adjusting the constant $c>0$.
\end{proof}
We end this section by noting that, thanks to the arm separation property Proposition \ref{prop:sep} (and by reasoning as in the proof of Lemma \ref{lem:ext}), we have the following result.
\begin{corollary}\label{cor:4arm}
There exists $c>0$ such that, for every $n \ge 1$ and every $x \in \Lambda_n$ that is at distance $\ge n/10$ from the sides of $\Lambda_n$, we have
\[
\Pro\big(\sigma,\sigma_t \in \mathsf{Piv}_x( \mathsf{Cross}_n) \big) \ge c\pi_n(t),
\]
where $(\sigma_s)_{s\ge 0}$ is the Glauber dynamics in $\Lambda_{2n}$ starting from $\mu_{2n}$.
\end{corollary}
\subsubsection{Proof of the quasi-multiplicativity property}
Before writing the proof, let us make three observations:
\begin{itemize}[noitemsep]
\item[(i)] Let $A_4^\mathrm{sep,in}(m,n)$ be an analogue of $A_4^\mathrm{sep}(m,n)$ but with arms well-separated at the inner scale $m$ (which is well defined for $n\ge 4m$ and $m \ge 100$, say). By the same proof as for Proposition~\ref{prop:sep},\footnote{Actually, there is one difference with the proof of Proposition \ref{prop:sep}. More precisely, for well-separation at the inner-scale, the $3$-arm event in the half-plane is not always sufficient to prove the analogue of Lemma~\ref{lem:3arm} because there is more room close to the corners. But by using the $3$-arm event in the half plane each time it is possible and using that the $1$-arm probability decays polynomially fast, one can obtain Lemma~\ref{lem:3arm} with $1-C\delta^c$ rather than $1-C\delta$, which is sufficient for our purpose.}, for every $n,m \ge 100$ that satisfy $n\ge 4m$ and for every $t\in[0,\tau]$, we have
\[
\pi^\mathrm{sep,in}_{m/4,n}(t) :=\Pro \big( \sigma,\sigma_t \in A_4^{\textup{sep,in}}(m/4,n) \big) \gtrsim \pi_{m,n}(t),
\]
where $(\sigma_s)_{s\ge 0}$ is the Glauber dynamics in $\Lambda_{2n}$ starting from $\sigma\sim\mu_{2n}$;
\item[(ii)] Moreover, for instance by finite-energy (Corollary \ref{cor:ins}), we have
\[
\pi_n(t) \gtrsim \pi^\mathrm{sep,in}_{100,n}(t)
\]
for every $n\ge 400$ and $t\in[0,\tau]$;
\item[(iii)] By the two points above and Proposition \ref{prop:sep}, we have $\pi_{m/4,4n}(t)\asymp \pi_{m,n}(t)$ for every $n \ge m \ge 4$ and $t\in[0,\tau]$.
\end{itemize}
\begin{proof}[Proof of Proposition \ref{prop:quasi_multi}]
We first note that, by these three items, it is sufficient to prove the quasi-multiplicativity property Proposition \ref{prop:quasi_multi} in the specific case where $n \ge 16m \ge 16^2k$ and $k$ is large enough. So let $k,m,n$ that satisfy these hypotheses and let $(\sigma_s)_{s\ge 0}$ be the Glauber dynamics in $\Lambda_{2n}$ starting from $\sigma\sim\mu_{2n}$.

\medskip

Let us first prove the first inequality of Proposition \ref{prop:quasi_multi}. By Proposition~\ref{prop:mix_fkg} (and by Proposition \ref{prop:quasi-inv} with $x=o$),
\[
\pi_{k,n}(t) \le \Pro \big( \sigma,\sigma_t \in A_4(k,m/4) \cap A_4(4m,n) \big) \lesssim \pi_{k,m/4}(t)\pi_{4m,n}(t).
\]
We conclude thanks to (iii) above.

\medskip

Let us now prove the second inequality of Proposition \ref{prop:quasi_multi}. We have
\[
\pi_{k,m}(t)\pi_{m,n}(t) \le \pi_{k,m/16}(t)\pi_{16m,n}(t) .
\]
As a result, by (i) above and Proposition \ref{prop:sep} (and by Proposition \ref{prop:quasi-inv} with $x=o$), it is sufficient to prove that
\[
\Pro \big( \sigma,\sigma_t \in A_4^{\textup{sep}}(k,m/4) \big) \Pro \big( \sigma,\sigma_t \in A_4^{\textup{sep,in}}(4m,n) \big) \lesssim \pi_{k,n}(t).
\]
This can be proven like Lemma \ref{lem:ext} (but by applying Proposition \ref{prop:mix_fkg} height times rather than four). This ends the proof.
\end{proof}
\subsection{Differential inequalities for percolation events}
In this section, we prove differential inequalities for (dynamical) percolation events by using the quasi-multiplicativity property Proposition \ref{prop:quasi_multi} and the differential formula Proposition \ref{prop:diff}.
\begin{lemma}\label{lem:diff_cross}
There exist $c,C>0$ such that, for any $n \ge 1$ and any $t\in[0,\tau]$,
\[
cn^2\pi_n(t)\le -\frac{d}{dt}\Pro(\sigma,\sigma_t \in \mathsf{Cross}_n) \le C \sum_{k=1}^n \frac{k^{1+c}}{n^c} \pi_k(t),
\]
where $(\sigma_s)_{s \ge 0}$ is the Glauber dynamics in $\Lambda_{2n}$ starting from $\sigma\sim\mu_{2n}$.
\end{lemma}
\begin{remark}
The presence of a sum on the right hand side comes from our estimate of contributions of the boundary vertices. However (at least when $\beta<\beta_c$, which is the case that will be useful for us), one also has that $-\frac{d}{dt}\Pro(\sigma,\sigma_t \in \mathsf{Cross}_n) \le C n^2\pi_n(t)$. This is not clear at this point but could be proven to be 
 via a more careful analysis of the boundary terms. We decided to omit this argument since the proof of the lemma above is simpler and is sufficient for our purpose.
 
\end{remark}
\begin{proof}
Now that we have proven the quasi-multiplicativity property (Proposition~\ref{prop:quasi_multi}) and interpreted differential formulas in terms of pivotal events (Proposition \ref{prop:diff}), the proof of this lemma is quite standard (see e.g.\ \cite[Chapter~6]{GS14}). We first notice that the lower bound is a direct consequence of Proposition \ref{prop:diff} and Corollary~\ref{cor:4arm}. So let us prove the upper bound. By Proposition~\ref{prop:diff}, it is sufficient to show that
\[
\sum_{x \in \Lambda_n} \Pro \big( \sigma,\sigma_t \in \mathsf{Piv}_x(\mathsf{Cross}_n) \big) \lesssim \sum_{k=1}^n \frac{k^{1+c}}{n^c} \pi_k(t).
\]
To deal with points that are quite close to the boundary, we need to consider the $3$-arm event in the half plane. We let $A_3^+(m,n) \subset \Omega_{2n}$ denote this (static) event. More precisely, this is the event that there are three paths of alternating sign from $\Lambda_m$ to $\partial \Lambda_n$, that are included in the (upper, say) half plane. Moreover, we let $\pi_{m,n}^{(3)}(t)=\Pro(\sigma,\sigma_t \in A_3^+(m,n))$.
It is known that one can compute the exponent of the $3$-arm event by using the box-crossing property \eqref{eq:bxp} (and \eqref{eq:smp} and \eqref{eq:mon}). Here we just need
\begin{equation}\label{eq:3arm}
\mu_{2n}\big(A_3^+(m,n)\big) \lesssim (m/n)^2,
\end{equation}
which we prove for completness in Section \ref{app:3arm} of the appendix.

\smallskip

We also let $A_1(m,n)\subset\Omega_{2n}$ denote the (static) $1$-arm event. More precisely, this is the event that there is a $+1$ path from $\Lambda_m$ to $\partial \Lambda_n$. We let $\pi_{m,n}^{(1)}(t)=\Pro(\sigma,\sigma_t\in A_1(m,n))$. By \eqref{eq:bxp}, there exists $c_0>0$ such that
\begin{equation}\label{eq:1arm}
\mu_{2n}\big(A_1(m,n)\big) \lesssim (m/n)^{c_0}.
\end{equation}
We fix such a constant $c_0>0$. We can (and we do) assume that $c_0<1$.    

\smallskip

Now, let $x \in \Lambda_n$, let $y$ denote the point on $\partial \Lambda_n$ closest to $x$ and let $z$ denote the corner of $\Lambda_n$ closest to $y$. Moreover, let $k$ be the distance between $x$ and $y$ and $j$ be the distance between $y$ and $z$. If $x$ is pivotal for $\mathsf{Cross}_n$ in some spin configuration, then the $4$-arm event holds from $x$ to $\partial \Lambda_k(x)$, the $3$-arm event holds in the `half-plane annulus' $(\Lambda_j(y)\setminus \Lambda_{2k}(y)) \cap \Lambda_n$ and the $1$-arm event holds in the annulus $\Lambda_{2n}(y)\setminus \Lambda_{2j}(y)$. Moreover, these three subsets are disjoint. As a result, the spatial mixing property Proposition \ref{prop:mix} (and the quasi-translation invariance property Proposition~\ref{prop:quasi-inv}) imply that
\[
\Pro \big( \sigma,\sigma_t \in \mathsf{Piv}_x(\mathsf{Cross}_n) \big) \lesssim \pi_{k/2}(t) \pi^{(3)}_{4k,j/2}(t) \pi^{(1)}_{2j,2n}(t).
\]
(Remark: the lengths $k/2$, $4k$, $j/2$, $2j$ appear to leave room between the box/annuli so that the spatial mixing property can indeed be applied.)

\medskip

By the quasi-multiplicativity property Proposition \ref{prop:quasi_multi}, \eqref{eq:3arm} and \eqref{eq:1arm} (and the fact that $\pi^{(3)}_{m,n}(t) \le \mu_{2n}(A_3^+(m,n))$ and similarly for the $1$-arm event), we have
\[
\Pro \big( \sigma,\sigma_t \in \mathsf{Piv}_x(\mathsf{Cross}_n) \big) \lesssim \pi_k(t) \bigg(\frac{k}{j}\bigg)^2 \bigg(\frac{j}{n}\bigg)^{c_0}.
\]
As a result,
\[
\sum_{x \in \Lambda_n} \Pro \big( \sigma,\sigma_t \in \mathsf{Piv}_x(\mathsf{Cross}_n) \big) \lesssim \sum_{k=1}^n \sum_{j=k}^n \pi_k(t) \bigg(\frac{k}{j}\bigg)^2 \bigg(\frac{j}{n}\bigg)^{c_0} \lesssim \sum_{k=1}^n \frac{k^{1+{c_0}}}{n^{c_0}} \pi_k(t). \qedhere
\]
\end{proof}
\begin{lemma}\label{lem:diff_arm}
There exists $C>0$ such that, for any $n \ge 1$ and $t\in[0,\tau]$,
\[
0\le -\pi_n'(t) \le C \pi_n(t) \sum_{k=1}^nk\pi_k(t).
\]
\end{lemma}
\begin{proof}
The proof is very similar to the proof of (the upper bound from) Lemma~\ref{lem:diff_cross}. The main difference is that the description of the pivotal events must take into account the multi-scale nature of the four arm event (see for instance \cite{Nol08} or \cite[proof of Lemmas~5.3 and 5.5]{TV23} for similar observations). Such a pivotal analysis implies that $-\pi_{m,n}'(t)$ is less than a constant times
\begin{equation}\label{eq:proof_diff_arm}
\pi_n(t) \bigg( \sum_{k=1}^n k\pi_k(t) + \sum_{k=1}^n n\pi_k(t)(k/n)^2 \bigg).
\end{equation}
We leave the details to the reader but we still describe briefly the two sums of the left-hand-side:
\begin{enumerate}
\item The first sum is the contribution of the vertices at distance $>n/10$ from $\partial \Lambda_n$. Here, $k$ can be thought of as the distance from $0$ to $x$;
\item The second sum is the contribution of the vertices that are close to $\partial \Lambda_n$. Here, $k$ can be thought of as the distance between $x$ and $\partial \Lambda_n$ and $(k/n)^2$ as the probability of a three-arm event in a half-space. (N.B: Contrary to the pivotal event for $\mathsf{Cross}_n$ -- see the proof of Lemma \ref{lem:diff_cross} --, corners do not play a special role here; this is why we obtain a bound of the kind $n\pi_k(t)(k/n)^2$ rather than $\sum_{j = k}^n\pi_k(t)(k/j)^2(j/n)^c$.)
\end{enumerate}
This ends the proof since \eqref{eq:proof_diff_arm} is less than $2\pi_n(t)\sum_{k=1}^nk\pi_k(t)$.
\end{proof}
\section{Sharp noise sensitivity for $\beta< \beta_c$}\label{sec:noise_sens}
Fix $\beta <\beta_c$ for the whole section (N.B: this implies that the `exponent' of the $4$-arm event is less than $2$, see \eqref{eq:4arm}). Let $\tau$ be the time constant as defined in Section~\ref{ssec:defin-time-insert}. 

\medskip

We now have all the intermediate results needed to adapt the proof of \cite{TV23} and prove the sharp noise sensitivity theorem (Theorem \ref{thm:sharp}). The only difference is that we have made our analysis only for times $t \in [0,\tau]$. We will use the dynamical mixing property (see Section \ref{ssec:dyn_mix}) to overcome this (as well as some monotonicity property in times). Let us finish the proof step by step, but by referring to \cite{TV23} when there is no difference at all.
\subsection{Superquadratic decay above the sensitivity length}
As in \cite{TV23}, we define the `noise sensitivity length' $\ell(t)$ as follows:
\begin{equation}\label{eq:sens_length}
\ell(t)=\min\{ n \ge 1 \, : \, n^2\alpha_n \geq 1/t \}.
\end{equation}
We refer to \cite[Sections 1.2 and 2.1]{TV23} for more discussions about this quantity. The general idea is that sharp noise sensitivity means that the percolation events are sensitive to an amount $t$ of noise at all scales $\gg \ell(t)$ but are stable under this amount of noise at all scales $\ll \ell(t)$. This is essentially the content of Theorem \ref{thm:sharp}, which is the main result of the present paper.

\medskip

In this section, we first notice that we have the following three properties, which are the analogues of Properties \textbf{(P1)}, \textbf{(P2)} and \textbf{(P3)} from \cite[Section 6]{TV23}. More precisely, there exists $C>0$ such that
\begin{itemize}
\item For every $n\ge 1$ and $0\le t\le u\le \tau$,
\begin{equation}\label{eq:P1}
  \tag{\textbf{P1}} 1\le \frac{\pi_n(t)}{\pi_n(u)}\le \exp\left(C\int_t^u\sum_{k=1}^n k \pi_k(s)ds\right);
\end{equation}
\item For every $n \ge 1$,
\begin{equation}\label{eq:P2}
  \int_0^\tau n^2 \pi_n(t) dt \le C;
\end{equation}
\item For every $n\ge m \ge 1$ and $0\le t\le u\le \tau$,
\begin{equation}\label{eq:P3}
\pi_n(u)\pi_m(t) \le C \pi_m(u)\pi_n(t).
\end{equation}
\end{itemize}
\begin{proof}
The first and second properties come from Lemmas \ref{lem:diff_arm} and \ref{lem:diff_cross} respectively. The third one is a consequence of the quasi-multiplicativity property Proposition~\ref{prop:quasi_multi} and the monotonicity of $s \mapsto \pi_{m,n}(s)$ (see the end of Section \ref{ssec:generator}):
\[
\frac{\pi_n(u)}{\pi_m(u)} \asymp \pi_{m,n}(u) \le \pi_{m,n}(t) \asymp  \frac{\pi_n(t)}{\pi_m(t)}. \qedhere
\]
\end{proof}

The following proposition is the analogue of \cite[Proposition 6.1]{TV23}, which is the main intermediate result from \cite{TV23}. 
\begin{proposition}\label{prop:superquadra}
There exist $c,C>0$ such that, for every $t \in [0,\tau]$ and every $n \ge m \ge \ell(t)$,
\[
\frac{\pi_n(t)}{\pi_m(t)} \le C\Big( \frac{m}{n} \Big)^{2+c}.
\]
\end{proposition}
\begin{proof}
Thanks to the three properties stated just before the proposition (and thanks to \eqref{eq:4arm}), the proof is exactly the same as in \cite[Section 6]{TV23} (Remark: The fact that the times are less than $\tau$ in these three properties does not change anything at all to the proof).
\end{proof}
\subsection{Proof of sharp noise sensitivity}
We can now prove the sharp noise sensitivity theorem. Recall that $\alpha_n=\mu_{2n}(A_4(n))$ and $\varepsilon_n=1/(n^2\alpha_n)$.
\begin{proof}[Proof of Theorem \ref{thm:sharp}]
Convention about the constants: In this proof, $c>0$ is a constant that may depend only on $\beta$ and that can change from one side of an inequality to the other.
Let $(\sigma_s)_{s\ge 0}$ denote the Glauber dynamics in $\Lambda_{2n}$ starting from $\sigma\sim\mu_{2n}$. 
 Let us first concentrate on the second part of the theorem, which is a consequence of Lemma~\ref{lem:diff_cross}. Indeed, if we assume that $t_n/\varepsilon_n \rightarrow 0$, then by this lemma we have
\begin{align}\label{eq:calcul_4arm}
0 \le 1/2-\Pro(\sigma,\sigma_t \in \mathsf{Cross}_n ) &\lesssim \int_0^{t_n}\sum_{k=1}^n \frac{k^{1+c}}{n^c}\pi_k(s)ds \nonumber\\
&\le t_n\sum_{k=1}^n \frac{k^{1+c}}{n^{c}}\alpha_k\nonumber\\
& \overset{\eqref{eq:4arm}}{\lesssim} t_n\alpha_n\sum_{k=1}^n \frac{k^{1+c}}{n^{c}}\Big(\frac{n}{k}\Big)^{2-c} \lesssim t_nn^2\alpha_n \underset{n\rightarrow +\infty}{\longrightarrow} 0.
\end{align}
Let us now concentrate on the first part of Theorem \ref{thm:sharp}, which is the main result of this paper and is a consequence of Lemma \ref{lem:diff_cross}, Proposition~\ref{prop:superquadra} and the dynamical mixing property \eqref{eq:dyn_mix}. Consider a sequence $(t_n)_n$ such that $t_n/\varepsilon_n\rightarrow +\infty$. By \eqref{eq:dyn_mix},
\begin{align*}
&\big| \Pro(\sigma,\sigma_t \in \mathsf{Cross}_n ) - 1/4 \big|\\
& \le \int_{t_n}^{\log n} -\frac{d}{ds} \Pro(\sigma,\sigma_s \in \mathsf{Cross}_n ) ds + n^{-c}\\
& \le \int_{t_n}^\tau -\frac{d}{ds} \Pro(\sigma,\sigma_s \in \mathsf{Cross}_n ) ds -\log n \frac{d}{dt} \Pro(\sigma,\sigma_t \in \mathsf{Cross}_n )_{|t=\tau} + n^{-c},
\end{align*}
because $t \mapsto \Pro(\sigma,\sigma_t \in \mathsf{Cross}_n )$ is convex (see the paragraph at the end of Section~\ref{sec:infin-gener}). By Lemma \ref{lem:diff_cross} and Proposition \ref{prop:superquadra} (applied to $t=\tau$ and $m=\ell(\tau)$), we have
\[
-\frac{d}{dt} \Pro(\sigma,\sigma_t \in \mathsf{Cross}_n )_{|t=\tau} \lesssim \sum_{k=1}^n \frac{k^{1+c}}{n^c}\pi_k(\tau) \lesssim \sum_{k=1}^n \frac{k^{1+c}}{n^{c}}k^{-(2+c)} \lesssim n^{-c}.
\]
(Remember that the constant $c$ can change from one side of an inequality to the other.)

As a result, and by using Lemma \ref{lem:diff_cross} once again, it is sufficient to prove that
\[
\int_{t_n}^\tau \sum_{k=1}^n \frac{k^{1+c}}{n^{c}} \pi_k(s)ds \underset{n\rightarrow +\infty}{\longrightarrow} 0.
\]
We show this as follows (our convention is that $\sum_{k=\ell(s)}^n$ equals $0$ if $n<\ell(s)$; we refer to the three remarks (a)--(b)--(c) just below for details about this computation):
\begin{align}\label{eq:last_term_of_the_paper}
\int_{t_n}^\tau \sum_{k=1}^n \frac{k^{1+c}}{n^{c}} \pi_k(s)ds  &\lesssim \int_{t_n}^{\tau} ds \bigg( \sum_{k=1}^{\ell(s)} \frac{k^{1+c}}{n^{c}} \pi_k(s) + \sum_{k=\ell(s)}^{n} \frac{k^{1+c}}{n^{c}} \pi_k(s) \bigg)\nonumber\\
& \hspace{-0.47cm}\overset{\textup{Prop. \ref{prop:superquadra}}}{\lesssim} \int_{t_n}^{\tau} ds \bigg( \sum_{k=1}^{\ell(s)} \frac{k^{1+c}}{n^{c}} \alpha_k + \sum_{k=\ell(s)}^{n} \frac{k^{1+c}}{n^{c}} \alpha_{\ell(s)} \Big( \frac{\ell(s)}{k} \Big)^{2+c} \bigg) \nonumber\\
& = \int_{t_n}^{\tau} ds \bigg( \frac{\ell(s)^c}{n^c} \sum_{k=1}^{\ell(s)} \frac{k^{1+c}}{\ell(s)^{c}} \alpha_k + \sum_{k=\ell(s)}^{n} \frac{k^{1+c}}{n^{c}} \alpha_{\ell(s)} \Big( \frac{\ell(s)}{k} \Big)^{2+c} \bigg) \nonumber\\
& \lesssim \int_{t_n}^{\tau} \frac{\ell(s)^{2+c} \alpha_{\ell(s)}}{n^c} ds.
\end{align}
Let us make three remarks about this computation: (a) In the second inequality we have also used that $\pi_k(s)\le \alpha_k$; (b) In the last inequality we have done the same computation as in \eqref{eq:calcul_4arm}; (c) It is important to remember that the constant $c$ can change from one side of an inequality to the other.

\smallskip

By definition of $\ell(s)$, we have $(\ell(s)-1)^2\alpha_{\ell(s)-1}\le\frac1 s \le \ell(s)^2\alpha_{\ell(s)}$, hence by the finite-energy property of static  Ising model, we get
\begin{equation}
  \label{eq:presque_defi_ell}
  \ell(s)^2\alpha_{\ell(s)}\asymp \frac 1s.
\end{equation}
Moreover, $\ell(\varepsilon_n)\asymp n$ by \eqref{eq:4arm}. As a result, \eqref{eq:last_term_of_the_paper} is of the same order as
\begin{equation}\label{eq:really_last_term_of_the_paper}
\int_{t_n}^\tau \frac{1}{s} \frac{\ell(s)^c}{\ell(\varepsilon_n)^c} ds.
\end{equation}
We now use that
\begin{equation}\label{eq:pol_ell}
\forall t \le s  \quad  \frac{\ell(s)}{\ell(t)} \lesssim \left(\frac t s\right) ^c.
\end{equation}
This can be proven for instance with $c=1/2$ by using \eqref{eq:presque_defi_ell} together with the fact that $\alpha_{\ell(t)}\le\alpha_{ \ell(s)}$. If we apply this to $t=\varepsilon_n$, we obtain that \eqref{eq:really_last_term_of_the_paper} is, up to constant, at most
\[
\int_{t_n}^{\tau} \frac{1}{s} \left( \frac{\varepsilon_n}{s} \right)^c ds \lesssim \left( \frac{\varepsilon_n}{t_n} \right)^c,
\]
which goes to $0$ as $n$ goes to $+\infty$ if $t_n/\varepsilon_n\rightarrow+\infty$.
\end{proof}
\appendix
\section{Appendix}
\subsection{The Ising model with boundary conditions}\label{sec:app_bc}
In this appendix, we recall the definition of the Ising model in a finite set with some boundary conditions $\xi$.
\begin{definition}\label{dfi:gen_ising}
Let $V$ be a finite subset of $\T$, let $E(V)$ denote the set of edges whose endpoints are both in $V$ and let $\xi$ be a boundary condition on $V$, i.e.\ $\xi \in \{-1,0,+1\}^{\partial_{\mathrm{ext}} V}$. The energy of a configuration $\eta\in\Omega_V:=\{-1,+1\}^V$ is defined by
\[
H^\xi(\eta)=-\bigg\{\sum_{\{x,y\}\in E(V)} \eta(x)\eta(y) + \sum_{x \in V, y \notin V \, : \, x \sim y} \eta(x)\xi(y) \bigg\}.
\]
Given some $\beta\in(0,+\infty)$, the Ising measure at inverse temperature $\beta$ on $V$ with boundary condition $\xi$ is the probability measure $\mu_V^\xi=\mu_{V,\beta}^\xi$ on $\Omega_V$ defined by
\[
\forall \eta\in\Omega_V, \quad \mu_V^\xi(\eta)=\frac{1}{Z_V^\xi}\exp\big(-\beta H^\xi(\eta)\big),
\]
where $Z_V^\xi=Z_{V,\beta}^\xi$ is the renormalizing constant $\sum_{\eta\in\Omega_V}\exp(-\beta H^\xi(\eta))$. We let $\langle\cdot\rangle_V^\xi$ denote the corresponding expectation.
\end{definition}
The critical inverse temperature is
\[
\beta_c = \sup \{ \beta \in (0,+\infty) : \, \<{\eta(o)}^+_{\Lambda_n} \rightarrow 0 \},
\]
where $o$ is the origin $(0,0)$ and $\langle\cdot\rangle_V^+=\langle\cdot\rangle_V^\xi$ with $\xi \equiv +1$. It is known that $\beta_c\in(0,+\infty)$ (see e.g.\ \cite{FV17}).
\subsection{Static percolation: the box-crossing property and two consequences}
\subsubsection{The box-crossing property}\label{sec:app_box}
In this section, we prove the box-crossing property \eqref{eq:bxp}, i.e.\ we prove that if $\beta\le \beta_c$, then for every $\rho>0$ and $\delta>0$, there exists $c_{\rho,\delta}>0$ such that, for every $n\ge 1$,
\[
\mu_{\Lambda^{(\delta)}_{\rho n,n}}^-(\mathsf{Cross}_{\rho n, n}) \ge c_{\rho,\delta},
\]
where $\Lambda^{(\delta)}_{\rho n,n}$ is the thickened elongated rhombus $\Lambda_{(1+\delta)\rho n,(1+\delta)n}$.
\begin{proof}[Proof in the case $\beta=\beta_c$]
When $\beta=\beta_c$, this is a consequence of the coupling between the Ising model (with $-1$ boundary conditions) and the FK--Ising model (with wired boundary conditions). We refer for instance to \cite{Gri06} for the definition of the FK--Ising model (also called random cluster model of parameter $q=2$) and for details about the coupling between Ising and FK--Ising. The following short proof is written as if the reader were familiar with this coupling.

To prove the result, we rely on the fact that the box-crossing-property (i.e.\ the analogue of \eqref{eq:bxp}) holds for critical FK--Ising percolation,  see e.g.\ \cite{DT19}. This implies that, for the FK--Ising model on $\Lambda_{\rho n,n}^{(\delta)}$ with wired boundary conditions, with non-negligible probability there exists a cluster that crosses $\Lambda_{\rho n,n}$ but does not touch the boundary of $\Lambda^{(\delta)}_{\rho n,n}$. If this cluster exists and that the sign $+1$ is assigned to it (which happens with conditional probability $1/2$), then the crossing event holds for the Ising model. This implies the desired result.
\end{proof}
\begin{proof}[Proof in the case $\beta<\beta_c$]
By the general Russo--Seymour--Welsh theory developed in \cite{KT23}, the result holds if the probability measure is replaced by the (unique) infinite volume Gibbs measure at inverse temperature $\beta$. Indeed, this probability measure satisfies the FKG inequality, is invariant under the symmetries of the triangular lattice, and the probability of $\mathsf{Cross}_n$ equals $1/2$ under this measure; as a result, \cite{KT23} implies the box-crossing property under this measure, see Theorem~1 therein.\footnote{The only difference is that \cite{KT23} considers measures on the edge configurations of $\Z^2$ that are invariant under the symmetries of $\Z^2$, but the proof in the case of site configurations on the triangular lattice is exactly the same.} By \eqref{eq:smp} and \eqref{eq:mon}, this implies the box-crossing property under $\mu_{\Lambda^{(\delta)}_{\rho n,n}}^+$. The following exponentially fast spatial mixing result ends the proof.
\end{proof}
\paragraph{Exponentially fast spatial mixing.} Assume that $\beta<\beta_c$. Then, there exists $C>0$ such that the following holds for every finite set $V\subset\T$ and every $k \ge 1$. Let $W$ be the set of vertices at (Euclidean) distance less than $k$ from $W$. Then,
\[
\sup_A \big| \mu_W^+(A) - \mu_W^-(A) \big| \le C|V|e^{-ck},
\]
when the supremum is on every $A \subseteq \Omega_W=\{-1,+1\}^W$ that depends only on $V$.
\begin{proof}
This result is (for instance) a consequence of (a) the coupling between the random cluster model of parameter $q=2$ and the Ising model (see e.g.\ \cite{Gri06}) and (b) the sharpness of the phase transition of the random cluster model (see e.g.\ \cite[Proposition~16]{DT19}). We leave the details of the proof to the reader.
\end{proof}
\begin{remark}\label{rem:fractal}
The box-crossing property is most likely true with $-1$ boundary conditions on $\Lambda_{\rho n,n}$, and even more generally for quads with `fractal' boundary (see \cite{DMT20} for such a result for random cluster models) but we have chosen to only prove \eqref{eq:bxp}, for simplicity and because it is enough for our purpose.
\end{remark}

\subsubsection{The $3$-arm event in the half-plane}\label{app:3arm}

In this section, we consider some $\beta\le\beta_c$. For every $n \ge m \ge 1$, let $A_3^+(m,n) \subset \Omega_{2n}$ denote the (static) $3$-arm event in the half-plane, that is, the event that there exist three paths of alternating sign from $\partial \Lambda_m$ to $\partial \Lambda_n$, that are included in the intersection of the annulus $\Lambda_n\setminus\Lambda_{m-1}$ with the upper-half-plane. It is quite standard that box-crossing properties and spatial Markov properties imply that the exponent of this event equals $2$. In this paper, we only use that the exponent is at least two, which is the following result, that we prove for completeness.
\begin{lemma}
There exists $C>0$ such that, for every $n \ge m \ge 1$,
\[
\mu_{2n}\big(A_3^+(m,n) \big)\le C\Big(\frac{m}{n}\Big)^2.
\]
\end{lemma}
\begin{proof}
We first notice that it is sufficient to consider the case $n \ge 100m$ and assume that this is indeed the case.
We let $\widehat{A}_3^+(m,n)$ denote the same event as $A_3^+(m,n)$ except that (a) we ask that the signs of the three paths are, when we go from left to right, $-$, $+$ and $-$ (in the definition of $A_3^+(m,n)$, this could also be $+$, $-$, and $+$) and (b) we just ask the two $-1$ paths to go from $\partial\Lambda_{2m}$ (rather than $\partial \Lambda_m$) to $\partial\Lambda_n$. We note that $\mu_{2n}(\widehat{A}_3^+(m,n))\ge \frac{1}{2}\mu_{2n}(A_3^+(m,n) )$.

\smallskip

We let $\widetilde{A}_3^+(m,n)$ denote the event $\widehat{A}_3^+(m,n)$ with the further property that there is a realization $\gamma_1$ of the $+1$ path and a realization $\gamma_0,\gamma_2$ of the two $-1$ paths such that (a) $\gamma_1$ is connected to the horizontal axis $\Z \times \{0\}$ by a $+1$ path included in $\Lambda_{2m}$ and (b) $\gamma_0$ is connected to $\gamma_2$ by a $-1$ path included in $\Lambda_{4m}\setminus \Lambda_{2m}$ (N.B: contrary to all the other paths in this definition, this path must escape the upper half-plane).

\smallskip

By using \eqref{eq:bxp}, \eqref{eq:smp} and \eqref{eq:mon}, one can prove that
\[
\mu_{2n}\big(\widetilde{A}_3^+(m,n)\;\big|\;\widehat{A}_3^+(m,n)\big) \gtrsim 1
\]
(we leave the details to the reader).

\smallskip

Next, for every $x \in \T$, we let $\widetilde{A}_3^+(x;m,n)$ denote the event $\widetilde{A}_3^+(m,n)$ but centered at $x$ instead of the origin $o$ and we let $D$ denote the set $\Lambda_{n/4} \cap \{k+\ell e^{i\pi/3}:(k,\ell) \in (8m+1)\Z^2\}$ (whose cardinality is of order $(n/m)^2$). Note that, if $x,y \in D$ are distinct, then $\Lambda_{4m}(x)\cap\Lambda_{4m}(y)=\emptyset$. Moreover, note that $\sum_{x\in D}1_{\widetilde{A}_3^+(x;m,n)}$ is at most the number of $+1$ clusters that cross the annulus $\Lambda_{n/2}\setminus\Lambda_{n/3}$ in the configuration restricted to $\Lambda_{n/2}$. By \eqref{eq:bxp} (and \eqref{eq:smp} and \eqref{eq:mon}), the expected number of such clusters (under $\mu_{2n}$) is uniformly bounded (see e.g.\ the arguments in \cite[Lemma A.2]{SS10}).
As a result, $\sum_{x\in D}\mu_{2n}(\widetilde{A}_3^+(x;m,n)) \lesssim 1$. The quasi-invariance by translation property~\eqref{eq:stat_mix_box} then implies that $|D|\mu_{2n}(\widetilde{A}_3^+(m,n))\lesssim 1$, which implies the desired result.
\end{proof}
\subsubsection{The `exponent' of the $4$-arm event is less than $2$}\label{sec:app_4arm}
In this section, we prove \eqref{eq:4arm}, i.e.\ we prove that if $\beta<\beta_c$ then there exists $c>0$ such that, for every $n\ge m \ge 1$,
\[
\frac{\alpha_n}{\alpha_m} \ge c\Big(\frac{m}{n}\Big)^{2-c}.
\]
\begin{proof}
We first note that it suffices to find $C_0>0$ such that \eqref{eq:4arm} holds for every $m,n\ge 1$ such that $m \ge C_0\log n$. So let $C_0>0$ to be chosen later and let $m,n\ge 1$ such that $m \ge C_0\log n$. We are inspired by \cite[Section~2.5 of the introduction]{Tas14}.
We tile $\Lambda_n$ with $\asymp (n/m)^2$ copies of $\Lambda_m$ and we let $(R_i)_{i\in I}$ denote the family of all those rhombii which are distance $>n/4$ from the sides of $\Lambda_n$. For every $i\in I$ we let $R_i'$ denote the rhombus whose center is the same $R_i$ and whose sides are twice larger. Moreover, for every $i\in I$,
\begin{itemize}
\item We let $\widetilde{A}_4(R_i',n)$ denote the event that there are four disjoint paths $\gamma_1,\dots,\gamma_4$ from $R_i'$ to $\partial \Lambda_n$, in counterclockwise order, whose endpoint belongs to the left, bottom, right and top side of $\Lambda_n$ respectively and such that $\gamma_j$ is made of $+1$ spins (resp.\ $-1$ spins) if $j$ is odd (resp.\ even);
\item We let $\widetilde{A}_5(R_i',n)$ denote the event that there exist $\gamma_1,\dots,\gamma_4$ as above as well as a fifth arm $\gamma_5$, made of $+1$ spins, also from $R_i'$ to $\partial \Lambda_n$, such that $\gamma_1,\dots,\gamma_5$ are still in counterclockwise order and such that the endpoint of $\gamma_5$ belongs to the top side of $\Lambda_n$ and the distance between $\gamma_1$ and $\gamma_5$ is larger than $C_0\log n$.
\end{itemize}
The desired result is a consequence of the following three observations, that hold if $C_0$ is large enough:
\begin{itemize}
\item By using the box-crossing property (and \eqref{eq:smp}, \eqref{eq:mon} and the exponentially fast spatial mixing property stated at the end of Section \ref{sec:app_box}), one can prove that $\mu_{2n}( \cup_{i\in I} \widetilde{A}_5(R_i',n) ) \gtrsim 1$.\footnote{To this purpose, one can condition on the lowest crossing of $\Lambda_n$ and on the event that this crossing is at distance $>n/3$ from the top and bottom sides of $\Lambda_n$, and then construct a $+1$ path and a $-1$ path from the top side of $\Lambda_n$ to the $C_0\log n$ neighbourhood of this lowest crossing, that are both at distance $>n/3$ from the left and right sides of $\Lambda_n$. See \cite[Section 2.5 of the introduction]{Tas14} for a similar construction, without the $C_0\log n$ buffer zone.} As a result,
\[
\sum_{i \in I} \mu_{2n}\big( \widetilde{A}_5(R_i',n) \big) \gtrsim 1.
\]
\item By conditioning on the lowest possible realization of $\gamma_1$ and the right-most realization of $\gamma_2$, and then using the box-crossing property at all scales between $m$ and $n$, one can prove that there exists $c>0$ such that\footnote{As above, we also use \eqref{eq:smp}, \eqref{eq:mon} and the exponentially fast spatial mixing property stated at the end of Section \ref{sec:app_box} -- here it is important that $\gamma_1$ and $\gamma_5$ are at distance $>C_0\log n$ from each other in the definition of $\widetilde{A}_5(R_i',n)$ to use the exponentially fast spatial mixing.}
\[
\mu_{2n}\big(\widetilde{A}_5(R_i',n) \; \big| \; \widetilde{A}_4(R_i',n) \big) \lesssim (m/n)^c.
\]
\item By the (static) quasi-translation invariance result (see Section~\ref{ssec:inv}) and the (static) separation of arms and quasi-multiplicativity properties (i.e.\ Propositions~\ref{prop:sep} and \ref{prop:quasi_multi} with $t=0$), we have
\[
\mu_{2n}\big(\widetilde{A}_4(R_i',n)\big) \asymp \alpha_n/\alpha_m. \qedhere
\]
\end{itemize}
\end{proof}
\subsection{Formulas related to the infinitesimal generator}
\subsubsection{The infinitesimal generator}\label{sec:app_gen}
In this section, we prove the formula for the infinitesimal generator from Section~\ref{ssec:generator}.
\begin{proof}
For every $t>0$ and every $x\in \Lambda_n$, we let $\mathcal{C}_t(x)$ denote the event that the clock at $x$ has rung exactly one time in the time interval $[0,t]$ (i.e.\ $\#\{Y_x\cap ([0,t]\times[0,1])\}=1$), that no other clock has rung in this time interval, and that when the clock has rung, the spin at $x$ has changed. For every $t>0$ and $\eta\in\Omega_n$, we have
\begin{align*}
\E\big[f(\sigma_t) \; \big| \; \sigma=\eta \big]-f(\eta) & = \sum_{x \in \Lambda_n} \E\big[1_{\mathcal{C}_t(x)}\big(f(\eta^x)-f(\eta) \big) \; \big| \; \sigma=\eta  \big] + O_{t\rightarrow 0}(t^2)\\
& = \sum_{x\in \Lambda_n} \Pro\big(\mathcal{C}_t(x)    \; \big| \; \sigma=\eta  \big) \big(f(\eta^x)-f(\eta)\big)+O_{t\rightarrow 0}(t^2).
\end{align*}
We conclude by observing that $\Pro(\mathcal{C}_t(x) \mid \sigma=\eta)=t c_x(\eta)+O_{t\rightarrow 0}(t^2)$.
\end{proof}
\subsubsection{A general differential formula}\label{sec:app_diff}
In this section, we prove the general differential formula stated in Section \ref{ssec:diff}, i.e.\
\[
-\frac{d}{dt}\E \big[f(\sigma)g(\sigma_t)\big] = \frac{1}{2}\sum_{x \in \Lambda_n} \E \Big[ c_x(\sigma) \big( f(\sigma^x)-f(\sigma) \big) \big( g(\sigma^{(x)}_t)-g(\sigma_t) \big) \Big],
\]
where $(\sigma_s)_{s\ge 0}$ is the Glauber dynamics starting from $\sigma\sim\mu_n$.
\begin{proof}
We let $(P_t)_{t\ge 0}$ denote the semi group of the Glauber dynamics in $\Lambda_n$, $L$ denote its infinitesimal generator (see Section \ref{ssec:generator}) and let $\langle\cdot\rangle_n$ denote the expectation with respect to $\mu_n$. We have
\begin{align*}
\frac{d}{dt}\E\big[f(\sigma)g(\sigma_t)\big] & = \frac{d}{dt}\big\langle fP_tg \big\rangle_n\\
& = \Big\{ \frac{d}{ds}\big\langle P_sfP_tg \big\rangle_n \Big\}_{|s=0},
\end{align*}
where in the second equality we used that, by reversibility of the dynamics, we have $\langle fP_{t+s}g \rangle_n=\langle P_sfP_tg \rangle_n$. As a result,
\begin{equation}\label{eq:diff2}
\frac{d}{dt}\E\big[f(\sigma)g(\sigma_t)\big] = \big\langle Lf P_t g \big\rangle_n  = \sum_{x\in\Lambda_n}\E\big[ c_x(\sigma) (f(\sigma^x)-f(\sigma)) g(\sigma_t) \big].
\end{equation}
We are now going to do an integration by parts. Fix some $x\in \Lambda_n$ and, for every $(\eta,\eta') \in \Omega_n\times\Omega_n$, let
\[
F(\eta,\eta')=c_x(\eta)(f(\eta^x)-f(\eta)) g(\eta').
\]
We first observe that (by applying the involution $\eta\mapsto\eta^x$ in the third equality):
\begin{align}\label{eq:diff3}
\E\big[ c_x(\sigma) (f(\sigma^x)-f(\sigma)) g(\sigma_t) \big] & = \E\big[F(\sigma,\sigma_t)\big]\nonumber\\
&=\sum_{\eta\in\Omega_n}\mu_n(\eta)\E\big[F(\sigma,\sigma_t)\; \big|\; \sigma=\eta\big]\nonumber\\
&= \sum_{\eta\in\Omega_n}\mu_n(\eta^x)\E\big[F(\sigma,\sigma_t)\; \big|\; \sigma=\eta^x\big]\nonumber\\
&= \sum_{\eta\in\Omega_n}\mu_n(\eta^x)\E\big[F(\sigma^x,\sigma_t^{(x)})\; \big|\; \sigma=\eta\big]\nonumber\\
&=\E\bigg[\frac{\mu_n(\sigma^x)}{\mu_n(\sigma)} F\big(\sigma^x,\sigma_t^{(x)}\big)\bigg]\nonumber\\
& = -\E\bigg[c_x(\sigma)\big(f(\sigma^x)-f(\sigma)\big) g(\sigma_t^{(x)}) \bigg]
\end{align}
where in the last equality we used that $\mu_n(\sigma^x)c_x(\sigma^x)=\mu_n(\sigma)$. As a result,
\begin{align*}
&\E\big[ c_x(\sigma) (f(\sigma^x)-f(\sigma)) g(\sigma_t) \big]\\
&\overset{\eqref{eq:diff3}}{=}\frac{1}{2} \Big\{ \E \big[ c_x(\sigma) (f(\sigma^x)-f(\sigma)) g(\sigma_t) \big] - \E \big[ c_x(\sigma) (f(\sigma^x)-f(\sigma)) g(\sigma_t) \big] \Big\}\\
& = \frac{1}{2} \E \Big[ c_x(\sigma) \big( f(\sigma^x)-f(\sigma) \big) \big( g(\sigma^{(x)}_t)-g(\sigma_t) \big) \Big].
\end{align*}
We conclude by combining the above with \eqref{eq:diff2}.
\end{proof}
\footnotesize
\bibliographystyle{alpha}
\bibliography{ref}

\begin{thebibliography}{AGMT16}

\bibitem[AB18]{AB18}
D.~Ahlberg and R.~Baldasso.
\newblock Noise sensitivity and {V}oronoi percolation.
\newblock {\em Electronic Journal of Probability}, 23, 2018.

\bibitem[ABGM14]{ABGM14}
D.~Ahlberg, E.~Broman, S.~Griffiths, and R.~Morris.
\newblock Noise sensitivity in continuum percolation.
\newblock {\em Israel Journal of Mathematics}, 201, 2014.

\bibitem[AGMT16]{AGMT16}
D.~Ahlberg, S.~Griffiths, R.~Morris, and V.~Tassion.
\newblock Quenched {V}oronoi percolation.
\newblock {\em Advances in Mathematics}, 286, 2016.

\bibitem[BGL13]{BGL13}
D.~Bakry, I.~Gentil, and M.~Ledoux.
\newblock {\em Analysis and geometry of {M}arkov diffusion operators}.
\newblock Springer Science \& Business Media, 2013.

\bibitem[BGS13]{BGS13}
E.I. Broman, C.~Garban, and J.E. Steif.
\newblock Exclusion sensitivity of {B}oolean functions.
\newblock {\em Probability theory and related fields}, 155, 2013.

\bibitem[BKS99]{BKS99}
I.~Benjamini, G.~Kalai, and O.~Schramm.
\newblock Noise sensitivity of {B}oolean functions and applications to
  percolation.
\newblock {\em Publications math{\'e}matiques de l'IH{\'E}S}, 90, 1999.

\bibitem[BS05]{BS05}
E.~Broman and J.E. Steif.
\newblock Dynamical stability of percolation for some interacting particle
  systems and {$\epsilon$}-stability.
\newblock {\em Annals of Probability}, 2005.

\bibitem[DM22]{DM22}
H.~{Duminil-Copin} and I.~Manolescu.
\newblock Planar random-cluster model: scaling relations.
\newblock {\em Forum of Mathematics, Pi}, 2022.

\bibitem[DMT21a]{DMT21}
H.~{Duminil-Copin}, I.~Manolescu, and V.~Tassion.
\newblock Near critical scaling relations for planar {B}ernoulli percolation
  without differential inequalities.
\newblock {\em arXiv preprint 2111.14414}, 2021.

\bibitem[DMT21b]{DMT20}
H.~{Duminil-Copin}, I.~Manolescu, and V.~Tassion.
\newblock Planar random-cluster model: fractal properties of the critical
  phase.
\newblock {\em Probability Theory and Related Fields}, 181, 2021.

\bibitem[DST17]{DST17}
H.~{Duminil-Copin}, V.~Sidoravicius, and V.~Tassion.
\newblock Continuity of the phase transition for planar random-cluster and
  potts models with {$1 \le q \le 4$}.
\newblock {\em Communications in Mathematical Physics}, 2017.

\bibitem[DT19]{DT19}
H.~{Duminil-Copin} and V.~Tassion.
\newblock Renormalization of crossing probabilities in the planar
  random-cluster model.
\newblock {\em Dobrushin's issue of Moscow Math.}, 2019.

\bibitem[FV17]{FV17}
S.~Friedli and Y.~Velenik.
\newblock {\em Statistical mechanics of lattice systems: a concrete
  mathematical introduction}.
\newblock Cambridge University Press, 2017.

\bibitem[GHSS21]{GHSS19}
C.~Garban, N.~Holden, A.~Sep{\'u}lveda, and X.~Sun.
\newblock Liouville dynamical percolation.
\newblock {\em Probability Theory and Related Fields}, 180, 2021.

\bibitem[GP24]{GP24}
P.~Galicza and G.~Pete.
\newblock Sparse reconstruction in spin systems {II}: {I}sing and other factor
  of {IID} measures.
\newblock {\em arXiv preprint 2406.09232}, 2024.

\bibitem[GPS10]{GPS10}
C.~Garban, G.~Pete, and O.~Schramm.
\newblock The {F}ourier spectrum of critical percolation.
\newblock {\em Acta Mathematica}, 205, 2010.

\bibitem[GPS18]{GPS18}
C.~Garban, G.~Pete, and O.~Schramm.
\newblock The scaling limits of near-critical and dynamical percolation.
\newblock {\em Journal of European Mathematical Society}, 20, 2018.

\bibitem[Gri99]{Gri99}
G.R. Grimmett.
\newblock {\em Percolation (Grundlehren der mathematischen Wissenschaften)}.
\newblock Springer, 1999.

\bibitem[Gri06]{Gri06}
G.~Grimmett.
\newblock {\em The random-cluster model}, volume 333.
\newblock Springer, 2006.

\bibitem[GS14]{GS14}
C.~Garban and J.E. Steif.
\newblock {\em Noise sensitivity of {B}oolean functions and percolation}.
\newblock Cambridge University Press, 2014.

\bibitem[GV19]{GV19a}
C.~Garban and H.~Vanneuville.
\newblock Exceptional times for percolation under exclusion dynamics.
\newblock {\em Annales Scientifiques de l'École normale supérieure}, 52,
  2019.

\bibitem[GV20]{GV19b}
C.~Garban and H.~Vanneuville.
\newblock Bargmann-{F}ock percolation is noise sensitive.
\newblock {\em Electronic Journal of Probability}, 25, 2020.

\bibitem[Kes87]{Kes87}
H.~Kesten.
\newblock Scaling relations for 2{D}-percolation.
\newblock {\em Communications in Mathematical Physics}, 109, 1987.

\bibitem[KT23]{KT23}
L.~{K{\"o}hler-Schindler} and V.~Tassion.
\newblock Crossing probabilities for planar percolation.
\newblock {\em Duke Mathematical Journal}, 172, 2023.

\bibitem[Lig85]{Lig85}
T.M. Liggett.
\newblock {\em Interacting particle systems}, volume~2.
\newblock Springer, 1985.

\bibitem[MOS94]{MOS94}
Fabio Martinelli, Enzo Olivieri, and Roberto~H. Schonmann.
\newblock For {$2$-$D$} lattice spin systems weak mixing implies strong mixing.
\newblock {\em Communications in Mathematical Physics}, 1994.

\bibitem[Nol08]{Nol08}
P.~Nolin.
\newblock Near-critical percolation in two dimensions.
\newblock {\em Electronic Journal of Probability}, 13, 2008.

\bibitem[RT24]{RT24}
R.~{Ramanan Radhakrishnan} and V.~Tassion.
\newblock Strict inequalities for arm exponents in planar percolation.
\newblock {\em arXiv preprint 2410.23250}, 2024.

\bibitem[SS10]{SS10}
O.~Schramm and J.E. Steif.
\newblock Quantitative noise sensitivity and exceptional times for percolation.
\newblock {\em Annals of Mathematics}, 171, 2010.

\bibitem[Tas14]{Tas14}
V.~Tassion.
\newblock {\em Planarit{\'e} et localit{\'e} en percolation}.
\newblock PhD thesis, Lyon, {\'E}cole normale sup{\'e}rieure, 2014.

\bibitem[TV23]{TV23}
V.~Tassion and H.~Vanneuville.
\newblock Noise sensitivity of percolation via differential inequalities.
\newblock {\em Proceedings of the London Mathematical Society}, 126, 2023.

\bibitem[Van21]{Van19}
H.~Vanneuville.
\newblock The annealed spectral sample of {V}oronoi percolation.
\newblock {\em The Annals of Probability}, 49, 2021.

\bibitem[Wer07]{Wer07}
W.~Werner.
\newblock Lectures on two-dimensional critical percolation.
\newblock {\em IAS Park City Graduate Summer School}, 2007.

\bibitem[Wu18]{Wu18}
H.~Wu.
\newblock Alternating arm exponents for the critical planar {I}sing model.
\newblock {\em The Annals of Probability}, 46, 2018.

\end{thebibliography}
\end{document}